\documentclass[a4paper,12pt]{amsart}
\usepackage{amsmath, amssymb, amsthm, xcolor, enumerate}
\usepackage{hyperref}
\usepackage{tikz-cd}
\usepackage[all]{xy}
\usepackage{cancel}
\usepackage[utf8]{inputenc}
\usepackage{csquotes}
\usepackage[english]{babel}
\usepackage{comment}

\newcounter{theorem}
\newtheorem{theorem}[theorem]{Theorem}
\newtheorem{lemma}[theorem]{Lemma}
\newtheorem{prop}[theorem]{Proposition}
\newtheorem{cor}[theorem]{Corollary}

\theoremstyle{definition}
\newtheorem{definition}[theorem]{Definition}

\newtheorem{question}[theorem]{Question}

\theoremstyle{remark}
\newtheorem*{remark*}{Remark}
\newtheorem{remark}[theorem]{Remark}

\numberwithin{equation}{section}
\allowdisplaybreaks

\newcommand{\Z}{\mathcal Z}
\newcommand{\C}{\mathrm{C}^*}

\newcommand{\linear}{\mathrm{linear}}

\newcommand{\Aff}{\mathrm{Aff}}

\newcommand{\Ad}{\operatorname{Ad}}

\newcommand{\KK}{\operatorname{KK}}
\newcommand{\UCT}{\operatorname{UCT}}

\newcommand{\nuc}{\operatorname{nuc}}
\newcommand{\Ann}{\operatorname{Ann}}

\title[]{Inclusions of real rank zero}
\author[J.\ Gabe]{James Gabe}
\address{\hskip-\parindent James Gabe, Department of Mathematics and Computer Science, University of Southern Denmark, Campusvej 55, 5230 Odense, Denmark}
\email{gabe@imada.sdu.dk}
\author[R.\ Neagu]{Robert Neagu}
\address{\hskip-\parindent Robert Neagu, Department of mathematics, KU Leuven, Celestijnenlaan 200B, 3001, Leuven, Belgium.}
\email{robert.neagu@kuleuven.be}

\begin{document}
\maketitle

    \begin{abstract}
We introduce a notion of real rank zero for inclusions of $\C$-algebras. After showing that our definition has many equivalent characterisations, we offer a complete description of the commutative case. We provide permanence and $K$-theoretic properties of inclusions of real rank zero and prove that a large class of full infinite inclusions have real rank zero. 
	\end{abstract}

\numberwithin{theorem}{section}

\section*{Introduction}
\renewcommand*{\thetheorem}{\Alph{theorem}}

The real rank of $\C$-algebras was introduced by Brown and Pedersen in \cite{rr0} as a non-commutative analogue of topological dimension. Particularly, the notion of real rank zero distinguished itself due to the fact that many interesting classes of $\C$-algebras were shown to have this property.\ To mention a few, simple purely infinite $\C$-algebras (\cite{zhangRR0}), irrational rotation algebras (\cite{IrrRotRR0,EllEv93}), AF-algebras (\cite{rr0}), and von Neumann algebras (\cite{rr0}) have real rank zero. Subsequently, real rank zero started to appear in the context of the classification programme of simple nuclear $\C$-algebras. Most notably in \cite{RR0Classif}, real rank zero facilitated access to classification by $K$-theory for large classes of $\C$-algebras.

Real rank zero was characterised in striking ways in \cite{rr0}, and is often thought of as the property of having a rich supply of projections. More precisely, a  $\C$-algebra has real rank zero exactly when all of its hereditary $\C$-subalgebras have an approximate unit consisting of projections.

On the other hand, the study of inclusions in operator algebras is ubiquitous in the last four decades. In the case of von Neumann algebras we can mention Jones' celebrated theory of subfactors (\cite{jonessubfactors}), or the Cartan subalgebras of Feldman--Moore (\cite{FeldmanMoore,FeldmanMooreII}). On the $\C$-algebraic side, inclusions appear in many concrete realisations, such as the AF core of a Cuntz-Krieger algebra (\cite{CuntzKrieger}) or Cartan pairs induced by an étale groupoid (cf.~\cite{Renault}), and more abstractly in Watatani's index theory (\cite{Watatani}) or Rørdam's $\C$-irreducible inclusions (\cite{RordamIrr}).\ 
 Inclusions also appear implicitly in the Elliott classification programme, as one often classifies inclusions of simple nuclear $\C$-algebras up to approximate unitary equivalence 
(\cite{O2stableclassif, ChrisClassif, oinftyclass, DynamicalKP, GLN23, classif}). 

We introduce the following notion of real rank zero for inclusions of $\C$-algebras.

\begin{definition}[Definition \ref{defn: RR0inclusions}]
We say that an inclusion of $\C$-algebras $A\subseteq B$ has \emph{real rank zero} if for any nonzero positive element $a\in A$, the hereditary $\C$-subalgebra $\overline{aBa}$ of $B$ has an approximate unit of projections.
\end{definition} 
In fact, as in the case of real rank zero for $\C$-algebras, the counterpart notion for inclusions enjoys a series of equivalent characterisations. For ease of notation, let us state the result in the unital case.

\begin{theorem}[Theorem \ref{thm: eqdefunital}]
Let $A\subseteq B$ be an inclusion of unital $\C$-algebras such that $1_B\in A$. Then the following are equivalent:
\begin{enumerate}
\item the inclusion $A\subseteq B$ has real rank zero;

\item any self-adjoint in $A$ can be approximated arbitrarily well by self-adjoints in $B$ with finite spectrum;

\item any self-adjoint in $A$ can be approximated arbitrarily well by invertible self-adjoints in $B$.
\end{enumerate}
\end{theorem}

For commutative $C^\ast$-algebras, real rank zero is characterised by the spectrum of the $\C$-algebra being  totally disconnected. The following gives a similar characterisation for inclusions of commutative $\C$-algebras.  

\begin{theorem}[Theorem \ref{thm: commutativecase}]\label{thm: IntroComm}
Let $A\subseteq B$ be an inclusion of unital commutative $\C$-algebras with $1_B\in A$. The inclusion $A\subseteq B$ has real rank zero if and only if there exists an intermediate commutative $\C$-algebra with totally disconnected spectrum.    
\end{theorem}

Furthermore, there are natural inclusions arising from dynamics which have real rank zero. The canonical inclusion $\C_r(G)\subseteq \ell^{\infty}(G)\rtimes_r G$ has real rank zero (Theorem \ref{roealgebraex}) for any countable discrete exact group $G$. This example is particularly intriguing since real rank zero seems to be quite a rare phenomenon both for group $\C$-algebras and for uniform Roe algebras.\footnote{We invite the reader to see the discussion in {\cite[Section 3]{RRRoeAlg}}. For example, the uniform Roe algebra of $SL(n,\mathbb{Z})$ does not have real rank zero if $n\geq 3$.} 

Next, we consider the class of inclusions recently classified in \cite{classif}.\ In the tracial setting, real rank zero of an inclusion between stabilised unital $\C$-algebras can be characterised by the canonical pairing between $K$-theory and traces. We refer the reader to Section \ref{sect: Kth} for the relevant definitions. 

\begin{theorem}[see Theorems \ref{thm: traceimpliesRR0} and \ref{thm: RR0impliestrace}]
Let $A\subseteq B$ be an inclusion of unital and separable $\C$-algebras such that $1_B\in A$.\ Suppose that $A$ and $B$ are simple, exact, stably finite, and $\Z$-stable.\ Then the inclusion $A\otimes\mathcal{K}\subseteq B\otimes\mathcal{K}$ has real rank zero if and only if the image of the induced map $\gamma^*:\Aff(T(A))\to\Aff(T(B))$ is contained in $\overline{\rho_B(K_0(B))}$.  
\end{theorem}

In the absence of traces, we consider inclusions which are simple and purely infinite in the sense that all non-zero positive elements of the subalgebra are full and properly infinite in the larger algebra. Hence we obtain a relative version of a classical result of Zhang (\cite{zhangRR0}) which says that simple purely infinite $\C$-algebras have real rank zero. 

\begin{theorem}[Theorem \ref{thm: purelyinfrr0}]\label{thm:IntroInf1}
Let $A\subseteq B$ be an inclusion of $\C$-algebras such that every nonzero positive element in $A$ is full and properly infinite in $B$.\ Then $A\subseteq B$ has real rank zero.
\end{theorem}

Motivated by the commutative setting (Theorem \ref{thm: IntroComm}), it seems fundamental to understand whether any inclusion of real rank zero (approximately) factors through a $\C$-subalgebra of real rank zero. We give a positive answer for the full inclusions classified by the first named author in \cite{oinftyclass}. 

\begin{theorem}[Theorem \ref{thm: factKirch}]
Let $A \subseteq B$ be an inclusion of separable exact $\C$-algebras such that $1_B \in A$ and suppose that the inclusion $\iota \colon A\hookrightarrow B$ is full, nuclear and $\mathcal{O}_\infty$-stable.\ Then $\iota$ is approximately unitarily equivalent to a $^*$-homomorphism which factors through a $\C$-algebra with real rank zero.
\end{theorem}

This paper is organised as follows. In Section \ref{sect: eqdef} we introduce real rank zero for inclusions and $^*$-homomorphisms. After showing a series of equivalent characterisations for this property, we offer a complete description of the commutative case in Section \ref{sect: CommCase}. Then, Section \ref{sect: permanenceprop} collects a number of permanence properties. In Section \ref{sect: Kth} we examine inclusions of real rank zero at the level of their invariant. In particular, we prove that under some regularity assumptions, real rank zero can be characterised by the relative pairing of $K$-theory and traces. Then, Section \ref{sect: purelyinf} focuses on infinite inclusions and provides some non-trivial examples of inclusions of real rank zero.

\subsection*{Acknowledgements} The first-named author was supported by the IRFD grants 1054-00094B and 1026-00371B. The second-named author was supported by the EPSRC grant EP/R513295/1. The first-named author would like to thank Matt Kennedy for suggesting applications to Furstenberg boundaries. The second-named author would like to thank Stuart White for his supervision on this project, Rufus Willett and Julian Kranz for helpful conversations on boundaries of hyperbolic groups, and Alistair Miller for asking a question that helped extending the result in Theorem \ref{roealgebraex}. This article forms part of the DPhil thesis of the second-named author.

\section{Equivalence of definitions}\label{sect: eqdef}

\numberwithin{theorem}{section}

Before defining real rank zero for inclusions, let us first recall the corresponding notion for $\C$-algebras, along with a number of equivalent characterisations. Precisely, in \cite{rr0}, Brown and Pedersen prove the following result.

\begin{theorem}[{\cite[Theorem 2.6]{rr0}}]\label{thm: RR0Alg}
For a $\C$-algebra $A$, the following conditions are equivalent:
\begin{enumerate}
    \item self-adjoint elements in $A$ can be approximated arbitrarily well by invertible self-adjoint elements in the minimal unitisation of $A$;
    \item the self-adjoint elements in $A$ with finite spectra are dense in the set of self-adjoint elements in $A$;
    \item every hereditary $\C$-subalgebra of $A$ has an approximate unit (not necessarily increasing) consisting of projections;
    \item for each pair of positive orthogonal elements $x,y$ in the minimal unitisation of $A$ and $\epsilon> 0$, there is a projection $p$ in the minimal unitisation of $A$, such that $\|(1-p)x\|\leq\epsilon$ and $\|py\|\leq\epsilon$.
\end{enumerate}
\end{theorem}

\begin{remark}
In \cite{rr0}, condition (i) of Theorem \ref{thm: RR0Alg} was taken as the definition of real rank zero for a $\C$-algebra. However, it is now common practice to say that a $\C$-algebra has real rank zero if it satisfies any of the equivalent conditions in Theorem \ref{thm: RR0Alg}. 
\end{remark}

We define the notion of real rank zero for inclusions of $\C$-algebras. However, we will show that, similar to the result in Theorem \ref{thm: RR0Alg}, there are a number of equivalent characterisations.

\begin{definition}\label{defn: RR0inclusions}
We say that an inclusion of $\C$-algebras $A\subseteq B$ has \emph{real rank zero} if for any nonzero positive element $a\in A$, the hereditary $\C$-subalgebra $\overline{aBa}$ of $B$ has an approximate unit of projections.\footnote{The approximate unit of projections is not assumed to be increasing.}
\end{definition}

\begin{remark}
Since $\overline{xBx^*}=\overline{xx^*Bxx^*}$ for all $x\in B$, an inclusion $A\subseteq B$ has real rank zero if and only if for any $a\in A$, $\overline{aBa}$ has an approximate unit of projections.
\end{remark}

A natural source of inclusions arises from $^*$-homomorphisms by considering the inclusion of the image into the codomain. Then, if this inclusion has real rank zero, we will say that the $^*$-homomorphism has real rank zero.

\begin{definition}\label{defn: RR0morphisms}
Let $\theta:A \to B$ be a $^*$-homomorphism between $\C$-algebras. Then we say that $\theta$ has \emph{real rank zero} if the inclusion $\theta(A)\subseteq B$ has real rank zero.
\end{definition}

In the spirit of {\cite[Theorem 2.6]{rr0}}, let us show that there are many equivalent characterisations of when an inclusion has real rank zero. Since some of the properties require the existence of a unit, let us first prove the following theorem in the unital case.

\begin{theorem}\label{thm: eqdefunital}
Let $A\subseteq B$ be an inclusion of unital $\C$-algebras such that $1_B\in A$. Then the following conditions are equivalent:
\begin{enumerate}
\item the inclusion $A\subseteq B$ has real rank zero;

\item for any self-adjoint $a\in A$ and any $\epsilon>0$, there exists a self-adjoint $b\in B$ with finite spectrum such that $\|a-b\|\leq\epsilon$;

\item for any self-adjoint $a\in A$ and any $\epsilon>0$, there exists an invertible self-adjoint $b\in B$ such that $\|a-b\|\leq \epsilon$;

\item for any positive $x,y\in A$ such that $xy=0$, and any $\epsilon>0$, there exists a projection $p\in B$ such that $\|px-x\|\leq\epsilon$ and $\|(1-p)y-y\|\leq\epsilon$.
\end{enumerate}
\end{theorem}

\begin{proof}

We are going to show that (iii) and (iv) are equivalent, and then that (i)$\implies$(iv)$\implies$(ii)$\implies$ (i).

Assume condition (iii). We will show that condition (iv) holds. Let $x,y\in A$ be positive and mutually orthogonal, $\epsilon>0$, and $\epsilon_0>0$ such that $\epsilon_0+\sqrt{(4+\epsilon_0)\epsilon_0)}<\epsilon$. If $x=0$, we take $p=0$ and if $y=0$, we take $p=1$. Therefore, we can assume that $x$ and $y$
are nonzero. Rescaling if necessary, suppose that $\|x\|=\|y\|=1.$ If we let $a=x-y$ which is self-adjoint, then by (iii), there exists $b\in B$ invertible and self-adjoint such that $\|a-b\|\leq\epsilon_0$. Let $p$ be the spectral projection of $b$ supported on $(0,\infty]$. Then, we note that $pb_+=b_+$ and $(1-p)b_-=b_-$.\footnote{Note that $b_+=\max(b,0)$ and $b_-=\max(-b,0)$. In particular, $b=b_+-b_-$.} We claim that $p$ satisfies the required conditions in (iv). Since $xy=0$, it follows that $a_+=(x-y)_+=x$. As $pb_+=b_+$, by an application of the triangle inequality, we have that
\begin{align*}
\|px-x\|&\leq\|pa_+-pb_+\|+\|a_+-b_+\|\\
&\leq 2\|a_+-b_+\|.
\end{align*} By {\cite[Lemma $2.2$]{rr0}}, we get that $2\|a_+-b_+\|\leq\delta+\epsilon_0$, where $\delta^2=(\|a\|+\|b\|)\epsilon_0$. Since $\|a\|\leq 2$, we have that $\|b\|\leq 2+\epsilon_0$ and hence $\delta^2\leq (4+\epsilon_0)\epsilon_0$. Therefore, $$\|px-x\|\leq \epsilon_0+ \sqrt{(4+\epsilon_0)\epsilon_0)}<\epsilon.$$ Similarly, as $a_-=y$ and $(1-p)b_-=b_-$, we get that $$\|(1-p)y-y\|\leq \|(1-p)a_--(1-p)b_-\|+\|a_--b_-\|.$$ The same argument as above gives that $\|(1-p)y-y\|\leq\epsilon$, which proves the claim.

Suppose that (iv) holds. To show (iii), let $a\in A$ be self-adjoint and $\epsilon>0$. Applying (iv) to $a_+$ and $a_-$, there is a projection $p\in B$ such that $\|(1-p)a_+\|\leq \epsilon$ and $\|pa_-\|\leq\epsilon$. 

Note that 
\begin{align*}
\|pa_+(1-p)+(1-p)a_+p\|& \leq \|a_+(1-p)\|+\|(1-p)a_+\|\\&\leq 2\epsilon.
\end{align*} Likewise, we get that $\|pa_-(1-p)+(1-p)a_-p\|\leq 2\epsilon$, so using that $a=a_+-a_-$, we obtain
\begin{equation}\label{eq8}
    \|a-(pap+(1-p)a(1-p))\|\leq 4\epsilon.
\end{equation}

Define a self-adjoint element in $B$ by $$b=pap+2\epsilon p +(1-p)a(1-p)-2\epsilon(1-p).$$ Since $\|pa_-p\|\leq\epsilon$, we have $pa_-p\leq \epsilon p$, so $pap\geq p(a_+-\epsilon)p\geq -\epsilon p$. Therefore, we get that $pbp\geq \epsilon p$, which shows that $pbp$ is invertible in $pBp$. Similarly, $(1-p)a(1-p)\leq \epsilon(1-p)$, so $(1-p)b(1-p)\leq -\epsilon(1-p)$, which gives that $(1-p)b(1-p)$ is invertible in $(1-p)B(1-p)$. Thus, $b$ is invertible in $B$, by an application of {\cite[Lemma 2.3]{rr0}} with $x=b$ and $c=d=0$. Finally, using \eqref{eq8}, a standard application of the triangle inequality shows that $\|a-b\|\leq 8\epsilon$.

We then move to the implication (i)$\implies$(iv). Let $x,y\in A$ positive orthogonal and $\epsilon>0$. Since $A\subseteq B$ has real rank zero, there exists a projection $p\in \overline{xBx}$ such that $\|px-x\|\leq\epsilon$. Since $xy=0$, we also get that $py=0$, so (iv) holds.

Suppose that (iv) holds. To show (ii), let $a\in A$ be self-adjoint and $\epsilon>0$. If $a=0$, the conclusion follows by taking $b=0$. Suppose that $a$ is nonzero. By replacing $a$ with $(\|a_+\|+\|a_-\|)^{-1}(a+\|a_-\|1_B)$, we can assume that $a\in A$ is a positive contraction. Indeed, if $(x_n)_{n\geq 1}$ is a sequence of self-adjoint elements in $B$ with finite spectrum converging to $(\|a_+\|+\|a_-\|)^{-1}(a+\|a_-\|1_B)$, then $(\|a_+\|+\|a_-\|)x_n-\|a_-\|1_B$ is a sequence of self-adjoint elements with finite spectrum converging to $a$.

Let $n\in\mathbb{N}$ such that $3/n<\epsilon$, so by applying condition (iv) to $(a-j/n)_+$ and $(a-j/n)_-$, for any $j\in\{1,2,\ldots,n-1\}$ there exists a projection $p_j\in B_{\infty}$\footnote{Here $B_\infty$ stands for the sequence algebra of $B$ given by $$B_\infty= \ell^{\infty}(\mathbb{N},B)/\Big\{(b_n)_{n\geq 1} : \lim\limits_{n\to\infty}\|b_n\|=0\Big\}.$$} such that $$p_j(a-j/n)_+=(a-j/n)_+ \ \text{and}\ (1-p_j)(a-j/n)_-=(a-j/n)_-.$$ Then, it follows that
\begin{equation}\label{eqn: computeproj}
p_ja=p_j((a-j/n)_+-(a-j/n)_-+(j/n)1_B)=(a-j/n)_++ (j/n)p_j.
\end{equation}Moreover, $p_ja=ap_j$.

Let $f_j\in C_0((0,1])_+$ be given by  \[  f_j(t)= 
\begin{cases} 
      0 & 0\leq t\leq \frac{j-1}{n} \\
      \linear & \frac{j-1}{n}\leq t\leq \frac{j}{n} \\
      1 & \frac{j}{n}\leq t\leq 1.
   \end{cases}
\] As the spectrum of $p_ja$ in $D_j:= C^*((a-j/n)_+,p_j)$ is contained in $[j/n,1]$ and $p_j$ commutes with $a$, we get that $p_j=f_j(p_ja)$. Moreover, since the map $\phi:\C(a)\to D_j$ given by $x\mapsto p_jx$
is a $^*$-homomorphism, we also get that $p_jf_j(a)=f_j(p_ja)=f_j(ap_j)=f_j(a)p_j$.

As $p_{j-1}$ acts as an identity on $(a-(j-1)/n)_+$, it also fixes $f_j(a)$, so $$p_{j-1}p_j=p_{j-1}f_j(a)p_j=f_j(a)p_j=p_j,$$ which gives that $p_1\geq p_2\geq\ldots\geq p_{n-1}$ is a decreasing sequence of projections. 

Let $q_0=1-p_1$, $q_j=p_j-p_{j+1}$ for $j=1,2,\ldots, n-2$ and $q_{n-1}=p_{n-1}$. Then, for $j=1,2,\ldots,n-2$ equation \eqref{eqn: computeproj} yields that $$q_ja=(a-j/n)_+-(a-(j+1)/n)_++(j/n)q_j-(1/n)p_{j+1}, $$so 
\begin{equation}\label{estimations1}
    \|q_ja-(j/n)q_j\|\leq 2/n.
\end{equation}Also, \begin{equation}\label{estimations2}
    \|q_0a\|=\|a-(a-1/n)_+-(1/n)p_1\|\leq 2/n
\end{equation} and 
\begin{equation}\label{estimations3}
    \left\|q_{n-1}a-\frac{n-1}{n}q_{n-1}\right\|=\left\|\left(a-\frac{n-1}{n}\right)_+\right\|\leq 1/n.
\end{equation}

Since all of the approximations above are in mutually orthogonal $\C$-algebras, putting together \eqref{estimations1}, \eqref{estimations2}, \eqref{estimations3}, and using that $\sum\limits_{j=0}^{n-1}q_j=1_{B_\infty}$, we get that \begin{equation}\label{estimations}
    \Bigg\|a-\sum\limits_{j=0}^{n-1}(j/n)q_j\Bigg\|=\Bigg\|\sum\limits_{j=0}^{n-1}(q_ja-(j/n)q_j)\Bigg\|\leq 2/n.
\end{equation}

Since $(q_j)_{j=0}^{n-1}$ is a sequence of mutually orthogonal projections that sum up to $1_{B_\infty}$, the map $\phi:\mathbb{C}^n\to B_\infty$ given by $(\lambda_0,\ldots,\lambda_{n-1})\mapsto \sum\limits_{j=0}^{n-1}\lambda_jq_j$ is a unital $^*$-homomorphism. Let $c_j(B)=\{(b_i)_{i\geq 1}\in\ell^{\infty}(B): b_i=0 \ \forall \ i\geq j\}$ and note that $(c_j(B))_{j\geq 1}$ is an increasing sequence of ideals of $\ell^\infty(B)$ generating $c_0(B)$. Then, since $\mathbb{C}^n$ is semiprojective (\cite[Lemma 14.1.5]{LoringLifting} and \cite[Theorem 14.2.1]{LoringLifting}), it follows that $\phi$ lifts to a unital $^*$-homomorphism $\tilde{\phi}:\mathbb{C}^n\to \ell^{\infty}(B)/c_j(B)\subseteq\ell^{\infty}(B)$ for some $j\geq 1$. Say that $\tilde{\phi}$ is given by a sequence of $^*$-homomorphisms $(\phi_k)_{k\geq 1}$. Then $x_k=\phi_k(0,1/n,\ldots, (n-1)/n)$ is a self-adjoint element with finite spectrum and \eqref{estimations} gives that $\limsup\limits_{k\to\infty}\|a-x_k\|\leq 2/n$. Thus, for $k$ large, we get that $\|a-x_k\|\leq 3/n<\epsilon$, so condition (ii) holds as claimed.

Suppose now that (ii) holds and we want to obtain (i). We remark for the reader's convenience that this part of the proof does not use the assumption that the inclusion is unital. Let $a\in A_+$ be nonzero, and $E$ be the hereditary $\C$-subalgebra of $B$ generated by $a$. Rescaling if necessary, we can assume that $\|a\|=1$. Given $0<\epsilon<1$, it suffices to find a projection $p\in E$ such that $\|a-pa\|\leq 2\epsilon$. Choose $\delta>0$ such that $9\delta < \epsilon - \epsilon^2$ and $n\in\mathbb{N}$ such that $1-\delta^{2/n} \leq \delta$.

By assumption, there exists $b\in B$ self-adjoint and with finite spectrum such that $\|a-b\|\leq \delta$. Since $a$ is positive of norm $1$, we can assume that $0\leq b \leq 1$. Moreover, the function $t \to t^{1/n}$ is continuous, so we may also assume that 
\begin{equation}
\|a^{1/n}-b^{1/n}\|\leq \delta.\label{eq1}
\end{equation}
If we denote by $q$ the spectral projection of $b$ corresponding to the interval $[\delta,1]$, then $\|b-qb\|\leq \delta$. Moreover, $q\in B$ as $b$ has finite spectrum. Recall that $1-\delta^{2/n} \leq \delta$. Since the spectrum of $bq$ in $qBq$ is contained in $[\delta,1]$, it follows that $$\|q-b^{2/n}q\|\leq \sup\limits_{\delta\leq t\leq 1}|1-t^{2/n}|\leq 1-\delta^{2/n}\leq \delta.$$ Using the triangle inequality, we now get that 
\begin{equation}
\|a-qa\|\leq \|b-qb\|+\|a-b\|+\|q(a-b)\|\leq 3\delta.\label{eq2}
\end{equation} Similarly, \eqref{eq1} and the estimate $\|q-b^{2/n}q\|\leq \delta$ give that 
\begin{equation}
\|a^{1/n}qa^{1/n}-q\|\leq 3\delta.\label{eq3}
\end{equation} With the notation $z=a^{1/n}qa^{1/n}$, the last inequality implies that $$\|z-z^2\|\leq\|z-q\|+\|z-q\|\|z+q\|\leq 9\delta$$ and by our choice of $\delta$, the spectrum of $z$ is contained in $[0,\epsilon]\cup [1-\epsilon,1]$. If we consider the spectral projection $p$ of $z$ corresponding to the interval $[1/2,1]$, then $p\in E$ as $z\in E$, and $\|z-p\|\leq \epsilon$.  Hence, 
$\|p-q\|\leq \epsilon+3\delta$ and thus $$\|a-pa\|\leq \epsilon+3\delta+ \|a-qa\|\leq \epsilon+6\delta<2\epsilon.$$
\end{proof}

We can now obtain a similar result for non-unital inclusions. For any $\C$-algebra $C$, we denote its forced unitisation by $C^{\dagger}$. 

\begin{theorem}\label{thm: eqdefnonunital}
Let $A\subseteq B$ be an inclusion of $\C$-algebras. Then the following are equivalent:
\begin{enumerate}
\item the inclusion $A\subseteq B$ has real rank zero;

\item for any self-adjoint $a\in A$ and any $\epsilon>0$, there exists a self-adjoint $b\in B$ with finite spectrum such that $\|a-b\|\leq\epsilon$;

\item for any self-adjoint $a\in A^{\dagger}$ and any $\epsilon>0$, there exists an invertible self-adjoint $b\in B^{\dagger}$ such that $\|a-b\|\leq \epsilon$;

\item for any positive $x,y\in A^{\dagger}$ such that $xy=0$, and any $\epsilon>0$, there exists a projection $p\in B^{\dagger}$ such that $\|px-x\|\leq\epsilon$ and $\|(1-p)y-y\|\leq\epsilon$.
\end{enumerate}    
\end{theorem}

\begin{proof}
Conditions (iii) and (iv) are equivalent by applying Theorem \ref{thm: eqdefunital} to the unital inclusion $A^{\dagger}\subseteq B^{\dagger}$. Suppose condition (i) holds and we claim it implies (iv). If $x,y\in A^{\dagger}$ are orthogonal, then at least one of $x$ and $y$ is in $A$. Suppose that $x\in A$ and let $p\in\overline{xBx}$ be a projection such that $\|px-x\|\leq\epsilon$. Since $xy=0$ and $p\in \overline{xBx}$, it also follows that $py=0$. 

Following the remark in the proof of Theorem \ref{thm: eqdefunital}, the fact that (ii) implies (i) is obtained in the proof of (ii)$\implies$(i) in Theorem \ref{thm: eqdefunital}. Then, it only remains to prove that (iv) implies (ii). Let $a\in A$ be self-adjoint, $\epsilon>0$ and suppose that (iv) holds. Then, applying Theorem \ref{thm: eqdefunital} to the inclusion $A^{\dagger}\subseteq B^{\dagger}$, there exists $b\in B^{\dagger}$ self-adjoint with finite spectrum such that $\|a-b\|\leq\epsilon/2$. Then $b=b_0+\lambda 1$, so $b_0\in B$ is self-adjoint, has finite spectrum, and $\|a-b_0\|\leq \epsilon$. Hence, the conclusion follows.
\end{proof}

\begin{remark}
In light of Theorem \ref{thm: eqdefnonunital}, we could have chosen any of the conditions above to characterise when an inclusion $A\subseteq B$ has real rank zero. However, we opted for the one appearing in Definition \ref{defn: RR0inclusions} because it can be stated both in the unital and non-unital case.
\end{remark}

We connect our definition to the usual notion of real rank zero for a $\C$-algebra, originally articulated in \cite{rr0}. Naturally, if at least one of the $\C$-algebras has real rank zero, then the inclusion has real rank zero.

\begin{lemma}\label{lemma: domainorcodRR0}
Let $A\subseteq B$ be an inclusion of $\C$-algebras. If there exists an intermediate $\C$-algebra with real rank zero, then the inclusion $A\subseteq B$ has real rank zero. In particular, if either $A$ or $B$ has real rank zero, then the inclusion $A\subseteq B$ has real rank zero. Moreover, a $\C$-algebra $C$ has real rank zero if and only if the inclusion $C\subseteq C$ has real rank zero.
\end{lemma}

\begin{proof}
Suppose first that there exists a $\C$-algebra $C$ with real rank zero such that $A\subseteq C\subseteq B$. We claim that the inclusion $A\subseteq B$ has real rank zero. For this, let $a\in A$ be self-adjoint and $\epsilon>0$. By Theorem \ref{thm: RR0Alg}, there exists a self-adjoint $b\in C$ with finite spectrum such that $\|a-b\|\leq \epsilon$. Since $b\in B$, we conclude that $A\subseteq B$ has real rank zero by Theorem \ref{thm: eqdefnonunital}. In particular, it follows that if $A$ or $B$ has real rank zero, then the inclusion $A\subseteq B$ has real rank zero. Therefore, if $C$ is a $\C$-algebra with real rank zero, the inclusion $C\subseteq C$ has real rank zero. Suppose now that the inclusion $C\subseteq C$ has real rank zero. To show that $C$ has real rank zero, let $c\in C$ be self-adjoint and $\epsilon>0$. Then, by Theorem \ref{thm: eqdefnonunital}, there exists a self-adjoint $d\in C$ with finite spectrum such that $\|c-d\|\leq \epsilon$. Hence, $C$ has real rank zero by Theorem \ref{thm: RR0Alg}. 
\end{proof}

\begin{lemma}\label{lemma: CompHomsRR0}
Let $A,B$, and $C$ be $\C$-algebras.\ Then the set of $^*$-homomorphisms from $A$ to $B$ with real rank zero is point-norm closed. Moreover, if $\phi:C\to B$ and $\psi:A\to C$ are $^*$-homomorphisms, at least one with real rank zero, then $\phi\circ\psi$ has real rank zero.
\end{lemma}

\begin{proof}
Let $\theta_n:A\to B$ be a sequence of $^*$-homomorphisms with real rank zero converging pointwise to $\theta:A\to B$. Then $\theta$ has real rank zero by using condition (ii) in Theorem \ref{thm: eqdefnonunital} for the sequence of inclusions $\theta_n(A)\subseteq B$. Suppose now $\phi:C\to B$ and $\psi:A\to C$ are $^*$-homomorphisms and $\psi$ has real rank zero. Let $a\in A$ be self-adjoint and $\epsilon>0$. Then there exists $c\in C$ self-adjoint with finite spectrum such that $\|\psi(a)-c\|\leq\epsilon$. Therefore, $\phi(c)$ is a self-adjoint element with finite spectrum and $\|\phi(\psi(a))-\phi(c)\|\leq \epsilon$. Hence $\phi\circ\psi$ has real rank zero. The case when $\phi$ has real rank zero follows similarly.
\end{proof}

\begin{lemma}\label{lemma: factRR0}
Let $A\subseteq B$ be an inclusion of $\C$-algebras. If the inclusion $A\subseteq B$ is approximately unitarily equivalent to a map factoring through a $\C$-algebra of real rank zero, then $A\subseteq B$ has real rank zero.
\end{lemma}

\begin{proof}
Let $C$ be a $\C$-algebra of real rank zero, $\psi:A\to C$ and $\phi:C\to B$ be $^*$-homomorphisms, and $(u_\lambda)_{\lambda}$ be a net of unitaries in the multiplier algebra of $B$ such that $\Ad(u_\lambda)\circ\phi\circ\psi$ converges pointwise to the inclusion $A\subseteq B$.

Since $C$ has real rank zero, combining Lemma \ref{lemma: domainorcodRR0} and Lemma \ref{lemma: CompHomsRR0} shows that the net of $^*$-homomorphisms $\Ad(u_\lambda)\circ\phi\circ\psi$ have real rank zero.\ Then, by the first part of the Lemma \ref{lemma: CompHomsRR0}, we get that the inclusion $A\subseteq B$ has real rank zero.
\end{proof}

\section{Inclusions of commutative $\C$-algebras}\label{sect: CommCase}

In this section we will give a complete description of when an inclusion of commutative $\C$-algebras has real rank zero.

\begin{theorem}\label{thm: commutativecase}
 Let $X,Y$ be compact Hausdorff spaces and $\phi:Y\to X$ be a continuous map which induces an inclusion $C(X)\subseteq C(Y)$ denoted by $\phi^*$.\ Then the inclusion $C(X)\subseteq C(Y)$ has real rank zero if and only if there exists an intermediate commutative $\C$-algebra with totally disconnected spectrum. 
\end{theorem}

\begin{proof}
The if direction follows from Lemma \ref{lemma: factRR0} and \cite[Proposition 1.1]{rr0}. Conversely, let us suppose that the inclusion has real rank zero. We first claim that $\phi$ is constant on each connected component of $Y$. To prove it, let $Y_0$ be a connected component of $Y$ and $y_1,y_2\in Y_0$. If $Y_0$ is a singleton, we are done, so let us assume that $y_1$ and $y_2$ are distinct. Suppose for a contradiction that $\phi(y_1)\neq \phi(y_2)$. By Urysohn's lemma, we then pick a self-adjoint $f\in C(X)$ such that $f(\phi(y_1))<0$ and $f(\phi(y_2))>0$. Let $\epsilon>0$ such that $2\epsilon<\min\{f(\phi(y_2)),|f(\phi(y_1))|\}$. 

Since $\phi^*$ has real rank zero, there exists an invertible self-adjoint $g\in C(Y)$ such that $|g(y)-\phi^*(f)(y)|\leq\epsilon$ for all $y\in Y$. In particular, $$g(y_1)\leq f(\phi(y_1))+\epsilon<0 \ \text{and} \ g(y_2)\geq f(\phi(y_2))-\epsilon>0.$$ Since $g$ is continuous and $Y_0$ is connected, $g(Y_0)\subseteq \mathbb{R}$ is connected, so there is $y_0\in Y_0$ such that $g(y_0)=0$.\ This contradicts the assumption that $g$ is invertible.\ Thus, $\phi(y_1)=\phi(y_2)$, which implies that $\phi$ is constant on $Y_0$.

Let $Y_c$ be the set of connected components of $Y$ equipped with the quotient topology. Since $\phi$ is constant on each connected component, there exists a continuous map $\psi:Y_c\to X$ given by evaluating $\phi$ on each connected component. Then $\phi=\psi\circ\pi$, where $\pi:Y\to Y_c$ is the canonical quotient map. Therefore, by duality, the diagram
\[
\begin{tikzcd}
    C(X) \arrow[rr, "\phi^*"] \arrow[swap, dr, "\psi^*"] & & C(Y) \\
    & C(Y_c) \arrow[swap, ur, "\pi^*"]
\end{tikzcd}
\]commutes, where $\psi^*$ and $\pi^*$ are the dual maps of $\psi$ and $\pi$ respectively. Since $Y$ is compact and Hausdorff, it follows that $Y_c$ is compact and totally disconnected. It is well-known that $Y_c$ is Hausdorff. To see this, note that in any compact Hausdorff space, connected components are quasicomponents and hence closed (see for example \cite[Proposition 1.4]{Top67}). As $Y$ is compact and Hausdorff, and in particular normal, any two connected components are contained in disjoint open sets. Consequently, the quotient space $Y_c$ is Hausdorff. Hence, $C(Y_c)$ has real rank zero by \cite[Proposition 1.1]{rr0}. Moreover, $\pi$ is surjective, which implies that $\pi^*$ is injective by duality. Hence $\pi^*(C(Y_c))$ has real rank zero.\ Since $C(X)\subseteq \pi^*(C(Y_c))\subseteq C(Y)$, we reach the conclusion.
\end{proof}

\begin{cor}\label{rmk: commmorphismsRR0}
 Let $X,Y$ be compact Hausdorff spaces and $\phi:Y\to X$ be a continuous map which induces a $^*$-homomorphism $\phi^*:C(X)\to C(Y)$.\ Then $\phi^*$ has real rank zero if and only if it factors through a commutative $\C$-algebra of real rank zero.
\end{cor}

\begin{proof}
The if direction follows from Lemma \ref{lemma: CompHomsRR0}. Conversely, suppose that $\phi^*$ has real rank zero. Note that $\phi^*(C(X))$ is a commutative $\C$-algebra, so we can identify it with some $\C$-algebra $C(Z)$ by Gelfand duality. Then the conclusion follows by applying Theorem \ref{thm: commutativecase} to the inclusion $C(Z)\subseteq C(Y)$ and \cite[Proposition 1.1]{rr0}.
\end{proof}

We use Corollary \ref{rmk: commmorphismsRR0} to characterise $^*$-homomorphisms between commutative $\C$-algebras which have \emph{nuclear dimension equal} to zero. Let us first record the following definition.

\begin{definition}(\cite{TW14,BGSW22})
Let $A$ and $B$ be $\C$-algebras and $\theta : A \to B$ be a $^*$-homomorphism. It is said that $\theta$ has \emph{nuclear dimension equal to zero} if for any finite set $\mathcal{F}\subseteq A$ and $\epsilon>0$, there exist a finite dimensional $\C$-algebra $F$ and cpc maps $$\psi :A\to F \ \text{and} \ \eta :F \to B$$ such that 
\begin{enumerate}
    \item $\|\eta(\psi(a))-\theta(a)\|\leq\epsilon$ for all $a\in\mathcal{F}$;
    \item $\eta$ is cpc order zero.\footnote{Recall that a completely positive and contractive (cpc) map $\phi: A\to B$ between $\C$-algebras is said to be \emph{order zero} if, for every $a,b \in A_+$ with $ab=0$, one has $\phi(a)\phi(b)=0$. The general theory for these maps was developed in \cite{orderzero}.}
\end{enumerate}
\end{definition}
Moreover, a $\C$-algebra $A$ is said to have nuclear dimension equal to zero if the identity map on $A$ has nuclear dimension equal to zero. It is known that a separable $\C$-algebra has nuclear dimension equal to zero if and only if it is approximately finite dimensional ({\cite[Theorem 4.2.3]{Win03}}).

\begin{prop}\label{prop: NucDimvsRR0}
Let $A,B$ be unital $\C$-algebras and $\theta:A\to B$ be a unital $^*$-homomorphism with nuclear dimension equal to zero. Then $\theta$ has real rank zero.
\end{prop}

\begin{proof}
Let $a\in A$ be a self-adjoint contraction and $\epsilon>0$. Building on \cite[Theorem 3.4]{Win03}, there exist a finite dimensional $\C$-algebra $F$, a ucp map $\psi:A\to F$, and a unital $^*$-homomorphism $\eta:F\to B$ such that $\|\theta(a)-\eta(\psi(a))\|< \epsilon/2$ (\cite[Lemma 3.1]{CN23}).\footnote{Even though \cite[Lemma 3.1]{CN23} is only stated in the setting when $A$ and $B$ are separable, this assumption is not needed.} Since $\psi(a)$ is self-adjoint and $F$ is finite dimensional, there exists a self-adjoint invertible $b'\in F$ such that $\|\psi(a)-b'\|\leq \epsilon/2$. Therefore, $\eta(b')$ is self-adjoint invertible and $$\|\theta(a)-\eta(b')\|\leq \|\theta(a)-\eta \circ \psi(a)\|+\|\psi(a)-b'\|\leq \epsilon.$$ Hence, $\theta$ has real rank zero by Theorem \ref{thm: eqdefunital}.
\end{proof}

\begin{theorem}
Let $X,Y$ be compact Hausdorff spaces and $\phi:Y\to X$ be a continuous map which induces a $^*$-homomorphism $\phi^*:C(X)\to C(Y)$. The following are equivalent:
\begin{enumerate}
    \item $\phi^*$ has nuclear dimension equal to zero;
    \item $\phi^*$ has real rank zero;
    \item $\phi^*$ factors through a commutative $\C$-algebra with nuclear dimension equal to zero.
\end{enumerate}
\end{theorem}

\begin{proof}
The fact that (i) implies (ii) follows from Proposition \ref{prop: NucDimvsRR0}. If $\phi^*$ has real rank zero, then it follows by Corollary \ref{rmk: commmorphismsRR0} that $\phi^*$ factors through a commutative $\C$-algebra of real rank zero. This algebra has nuclear dimension equal to zero by \cite[Proposition 1.1]{rr0} and \cite[Proposition 2.4]{nucdim} (see also \cite[Theorem 2.6]{NucDimJorge}). The fact that (iii) implies (i) is immediate.
\end{proof}

We end this section by applying Theorem \ref{thm: commutativecase} to obtain a class of $^*$-homomorphisms with real rank zero.

\begin{theorem}\label{thm: RR0CrossedP}
Let $G$ be a compact connected group, $X,Y$ be compact Hausdorff $G$-spaces and $\phi:Y\to X$ be an equivariant continuous map. If $\phi^*:C(X)\to C(Y)$ has real rank zero, then the induced morphism $\tilde{\phi}:C(X)\rtimes_r G\to C(Y)\rtimes_r G$ factors through a $\C$-algebra of real rank zero, and in particular $\tilde{\phi}$ has real rank zero.
\end{theorem}

\begin{proof}
Let $Y_c$ be the set of connected components of $Y$ equipped with the quotient topology. Since $\phi^*$ has real rank zero, the proof of Theorem \ref{thm: commutativecase} shows that there exists a continuous map $\psi:Y_c\to X$ given by evaluating $\phi$ on each connected component.\ Then $\phi=\psi\circ\pi$, where $\pi:Y\to Y_c$ is the canonical quotient map.\ We equip the space $Y_c$ with the trivial action of $G$ and claim that the maps $\pi$ and $\psi$ are $G$-equivariant.

First note that since $G$ is connected, the orbit $Gy$ is connected for any $y\in Y$. Let $g\in G$ and $y\in Y$. Then, since $g\cdot y$ and $y$ are in the same connected component of $Y$, we have that $$\pi(g\cdot y)=\pi(y)=g\cdot \pi(y).$$ Hence, the map $\pi$ is $G$-equivariant.

Similarly, since $g\cdot y$ and $y$ are in the same connected component of $Y$, the proof of Theorem \ref{thm: commutativecase} shows that $\phi(g\cdot y)=\phi(y)$. Therefore, using the equivariance of $\phi$, one gets that $$\psi(g\cdot\pi(y))=\psi(\pi(y))=\phi(y)=\phi(g\cdot y)=g\cdot \phi(y)=g\cdot\psi(\pi(y)).$$ Hence, the $^*$-homomorphism $\phi^*$ factors equivariantly as \[
\begin{tikzcd}
    C(X) \arrow[rr, "\phi^*"] \arrow[swap, dr, "\psi^*"] & & C(Y) \\
    & C(Y_c) \arrow[swap, ur, "\pi^*"].
\end{tikzcd}
\]

Therefore, we obtain the commuting diagram 
\[
\begin{tikzcd}
    C(X)\rtimes_r G \arrow[rr, "\tilde{\phi}"] \arrow[swap, dr] & & C(Y)\rtimes_r G \\
    & C(Y_c)\rtimes_r G \arrow[swap, ur].
\end{tikzcd}
\] Since $G$ acts trivially on $Y_c$, we have that $C(Y_c)\rtimes_r G\cong C(Y_c)\otimes \C_r(G)$. Then, $G$ is compact connected, so $\C_r(G)$ has real rank zero by {\cite[Theorem 1]{GroupsRR0}}. Moreover, $C(Y_c)$ has real rank zero by \cite[Proposition 1.1]{rr0}. In particular, $C(Y_c)$ is an AF-algebra. Then, $C(Y_c)\otimes \C_r(G)$ has real rank zero by \cite[Theorem 3.2]{rr0} and we reach the conclusion. 
\end{proof}

\section{Permanence properties}\label{sect: permanenceprop}

We first show that real rank zero is preserved under inductive limits. 

\begin{prop}\label{prop: rr0inductivelim}
Suppose that $A=\lim\limits_{\longrightarrow}(A_n,\phi_n)$ and $B=\lim\limits_{\longrightarrow}(B_n,\psi_n)$ are two $\C$-inductive limits and there exist inclusions $A_n\subseteq B_n$ for all $n\in\mathbb{N}$, such that $\psi_n|_{A_n}=\phi_n$.\ If there exists $N\in\mathbb{N}$ such that $A_n\subseteq B_n$ has real rank zero for all $n>N$, then $A\subseteq B$ has real rank zero.
\end{prop}

\begin{proof}
Since $\psi_n|_{A_n}=\phi_n$ for any $n\in\mathbb{N}$, the sequence of inclusions $A_n\subseteq B_n$ induce an inclusion $A\subseteq B$.\ If the inclusion $A\subseteq B$ is non-unital, then consider the unital inclusion $A^{\dagger}\subseteq B^{\dagger}$ with unital connecting maps.\ If this inclusion has real rank zero, then by Theorem \ref{thm: eqdefnonunital}, the inclusion $A\subseteq B$ has real rank zero. Therefore, we can assume that $A\subseteq B$ is a unital inclusion, with $A_n,B_n$ unital for all $n\geq 1$ and unital connecting maps. 

Let $a\in A$ be a self-adjoint element and $\epsilon>0$. Then there exist $n>N$ and a self-adjoint element $a_n\in A_n$ such that $\|a-\phi_{n,\infty}(a_n)\|\leq \epsilon/2$, where $\phi_{n,\infty}:A_n\to A$ is an inductive limit connecting map.\ Since $\psi_k|_{A_k}=\phi_k$ for any $k\in\mathbb{N}$, we get that $\phi_{n,\infty}(a_n)=\psi_{n,\infty}(a_n)$, so $$\|a-\psi_{n,\infty}(a_n)\|\leq\epsilon/2.$$ By assumption, $A_n\subseteq B_n$ has real rank zero, so there exists an invertible self-adjoint $b_n\in B_n$ such that $\|a_n-b_n\|\leq\epsilon/2$. Since $\psi_{n,\infty}$ is a unital $^*$-homomorphism, $\psi_{n,\infty}(b_n)$ is a self-adjoint invertible element.\ Then, triangle inequality gives that $\|a-\psi_{n,\infty}(b_n)\|\leq\epsilon$.\ Hence $A\subseteq B$ has real rank zero by Theorem \ref{thm: eqdefunital}.
\end{proof}

Then, we can show that real rank zero of an inclusion is preserved if we consider an intermediate inclusion into a hereditary subalgebra. 

\begin{prop}\label{prop: rrohereditary}
Let $A\subseteq B$ be an inclusion of $\C$-algebras and $D$ be a hereditary subalgebra of $B$ such that $A\subseteq D$. If $A\subseteq B$ has real rank zero, then the inclusion $A\subseteq D$ has real rank zero.
\end{prop}

\begin{proof}
Let $a\in A_+$ be nonzero.\ Then $\overline{aBa}$ has an approximate unit of projections.\ Since $\overline{aDa}=\overline{aBa}$, the inclusion $A\subseteq D$ has real rank zero.
\end{proof}

\begin{cor}\label{cor: rr0hereditary}
Let $A\subseteq B$ be an inclusion of $\C$-algebras and $a\in A$ be a positive element.\ If $A\subseteq B$ has real rank zero, then the induced inclusion $\overline{aAa}\subseteq\overline{aBa}$ has real rank zero.
\end{cor}

\begin{proof}
Since $A\subseteq B$ has real rank zero, the inclusion $\overline{aAa}\subseteq B$ also has real rank zero.\ Hence, $\overline{aAa}\subseteq\overline{aBa}$ has real rank zero by Proposition \ref{prop: rrohereditary}.
\end{proof}

We will now examine the behaviour of inclusions of real rank zero under extensions. The next three propositions follow the proof of {\cite[Theorem 3.14]{rr0}}.

\begin{prop}\label{prop: rr0extensions1}
Let $A,B$ be $\C$-algebras and $I$ and $J$ be closed $2$-sided ideals of $A$ and $B$ respectively.\ Suppose that $A\subseteq B$, $I\subseteq J$, and the inclusion $A\subseteq B$ has real rank zero.\ Denote the induced map between the quotients by $\theta:A/I\to B/J$.\ Then $I\subseteq J$ and $\theta(A/I)\subseteq B/J$ have real rank zero.\ Moreover, for any nonzero positive element $x\in A$, any projection in $\overline{xAx}/\overline{xIx}$ lifts to a projection in $\overline{xBx}$.
\end{prop}

\begin{proof}
Let $a\in I_+$ and $\epsilon>0$. Since $\overline{aBa}$ has an approximate unit of projections, there exists a projection $p\in \overline{aBa}$ such that $\|pa-a\|\leq \epsilon.$ However, $a\in J$, so $p\in \overline{aJa}$. Thus, $I\subseteq J$ has real rank zero.

Let us now consider the inclusion $\theta(A/I)\subseteq B/J$.\ Let $a\in A$ be a self-adjoint element and $\epsilon>0$.\ Then, by Theorem \ref{thm: eqdefnonunital}, there exists a self-adjoint element with finite spectrum $b\in B$ such that $\|a-b\|\leq \epsilon$. But $b+J$ is still a self-adjoint element with finite spectrum and $$\|\theta(a+I)-(b+J)\|=\|(a-b)+J\|\leq \epsilon.$$ Hence, $\theta(A/I)\subseteq B/J$ has real rank zero.

Let $x\in A_+$ be nonzero.\ By Corollary \ref{cor: rr0hereditary}, the inclusion $\overline{xAx}\subseteq\overline{xBx}$ has real rank zero.\ Therefore, replacing the inclusion $A\subseteq B$ with $\overline{xAx}\subseteq\overline{xBx}$, it suffices to check that any nonzero projection in $\theta(A/I)$ lifts to a projection in $B$.\ Let $a\in A$ be a self-adjoint element of norm $1$ such that $a+J$ is a nonzero projection and let $0<\epsilon<1$.\ Then there exists $b\in B$ self-adjoint with finite spectrum and norm $1$ such that $\|a-b\|\leq\epsilon^2$. Since $b$ is a finite linear combination of its spectral projections, we can write $b=\sum\limits_{k=1}^n\lambda_kp_k$, for some $n\in\mathbb{N}$, $\lambda_k\in\mathbb{R}$, and $p_k\in B$ pairwise orthogonal projections. 

Let us denote by $\pi:B\to B/J$ the canonical quotient map.\ Using that $\pi(a)$ is a projection, a direct estimation yields that 
\begin{equation}
    \left\|\pi\left(\sum\limits_{k=1}^n(\lambda_k-\lambda_k^2)p_k\right)\right\|=\|\pi(b-b^2-a+a^2)\|\leq 3\epsilon^2.
\end{equation} Therefore, $p_k\in J$ whenever $|\lambda_k-\lambda_k^2|\geq 4\epsilon^2$. 

Let $S$ and $T$ be the subsets of $\{1,2,\ldots,n\}$ consisting of those $k$ such that $|1-\lambda_k|<2\epsilon$ and $|\lambda_k|<2\epsilon$, respectively. Setting $p=\sum\limits_{k\in S}p_k$ and observing that $p_k\in J$ if $k$ is not contained in the union of $S$ and $T$, it follows that $$\|\pi(p-b)\|\leq \Bigg\|\sum\limits_{k\in S}p_k-\sum\limits_{k\in S\cup T}\lambda_kp_k\Bigg\|\leq 2 \epsilon.$$ Then, $\|\pi(p-a)\|\leq 2\epsilon+\epsilon^2$, so there is a positive contraction $e\in J$ such that $\| (1-e)(p-a) \| \leq 3\epsilon + \epsilon^2$. Hence, $z= (1-e)a(1-e) + ep+pe-epe \in B$ is self-adjoint such that $z-a \in J$
 and $\| p-z \| \leq 3 \epsilon + \epsilon^2$.
So for $\epsilon$ small enough we have $\|p-z\|<\frac{1}{16}$. Since $\|z+p\|\leq 2+\frac{1}{16}$, it follows that $$\|z-z^2\|=\|(z-p)-(z^2-p^2)\|< \frac{1}{4},$$ so the spectrum of $z$ has a gap around $\frac{1}{2}$.\ Let 
 \[  f(t)= 
\begin{cases} 
      0 & t<\frac{1}{2} \\
      1 & t>\frac{1}{2}
   \end{cases}
\] be a continuous function on the spectrum of $z$.\ Then $f(z)$ is a projection and it is a lift of the projection $a+J$.\ Hence, for any positive nonzero $x\in A$, any projection in $\overline{xAx}/\overline{xIx}$ lifts to a projection in $\overline{xBx}$.
\end{proof}

\begin{remark}
Let us reflect on the converse of the above proposition. Even if the inclusions $I\subseteq J$ and $\theta(A/I)\subseteq B/J$ have real rank zero, and any projection in $\theta(A/I)$ lifts to a projection in $B$, we cannot conclude that $A\subseteq B$ has real rank zero.\ For example, if $I=\{0\}$ and $J=B$, then $I\subseteq J$ and $\theta(A/I)\subseteq B/J$ have real rank zero.\ Moreover, any projection in $\theta(A/I)$ lifts to a projection in $B$.\ However, the fact that $A\subseteq B$ has real rank zero does not hold in general.
\end{remark}

Let us now provide partial converses to Proposition \ref{prop: rr0extensions1}.

\begin{prop}\label{prop: rroextensions2}
Let $A,B$ be $\C$-algebras and $I$ and $J$ be closed $2$-sided ideals of $A$ and $B$ respectively such that $A\subseteq B$ and $I\subseteq J$. Suppose that $A/I$ has real rank zero and the inclusion $I\subseteq J$ has real rank zero.\ Moreover, assume that for any hereditary subalgebra $D$ of $A$, any projection in $(D+I)/I$ lifts to a projection in $D$.\ Then $A\subseteq B$ has real rank zero.
\end{prop}

\begin{proof}
Let $x,y\in A^{\dagger}$ be positive orthogonal elements and $\epsilon>0$. Since $xy=0$, at least one of them is in $A$, so let us assume that $x\in A$. Let $C=\overline{xAx}$ and $\pi:A\to A/I$ be the canonical quotient map.\ Then $\pi(C)$ is the hereditary subalgebra of $A/I$ generated by $\pi(x)$, so it has real rank zero. Therefore, there exists a projection $q\in \pi(C)$ such that 
\begin{equation}\label{eq: PermProp1}
\|q\pi(x)-\pi(x)\|\leq\epsilon.
\end{equation}By assumption, we can lift $q$ to a projection $p\in C$.

Consider $(1-p)\overline{xIx}(1-p)$ and $(1-p)\overline{xJx}(1-p)$.\ Then, since the inclusion $I\subseteq J$ has real rank zero, a similar proof to that of Proposition \ref{prop: rrohereditary} shows that the inclusion $(1-p)\overline{xIx}(1-p)\subseteq (1-p)\overline{xJx}(1-p)$ has real rank zero.\ Moreover, \eqref{eq: PermProp1} gives $\|\pi(x-px)\|\leq\epsilon$, so we can find a projection $r\in(1-p)\overline{xJx}(1-p)$ such that
\begin{equation}\label{eq: PermProp2}
\|(x-px)-r(x-px)\|\leq2\epsilon.
\end{equation} Note that $p$ and $r$ are orthogonal, so let $p_1=p+r$, which is a projection in $\overline{xB^{\dagger}x}$. Since $xy=0$, it follows that $p_1y=0$. Then, \eqref{eq: PermProp2} yields that $\|x-p_1x\|\leq2\epsilon$, so $A\subseteq B$ has real rank zero by Theorem \ref{thm: eqdefnonunital}(iv).
\end{proof}

\begin{prop}\label{prop: rr0extension3}
Let $A,B$ be $\C$-algebras and $I$ and $J$ be closed $2$-sided ideals of $A$ and $B$ respectively such that $A\subseteq B$ and $I\subseteq J$. Denote the induced map between the quotients by $\theta:A/I\to B/J$.\ Suppose that $J$ has real rank zero, any projection in $B/J$ lifts to a projection in $B$, and $\theta(A/I)\subseteq B/J$ has real rank zero.\ Then $A\subseteq B$ has real rank zero.  
\end{prop}

\begin{proof}
Let $x,y\in A^{\dagger}$ be positive orthogonal elements and $\epsilon>0$. Since $xy=0$, at least one of them is in $A$, so let us assume that $x\in A$. Let $C=\overline{xAx}$ and $\pi:A\to A/I$ the canonical quotient map.\ Then $\theta(\pi(C))$ is the hereditary subalgebra of $B/J$ generated by $x+J$.\ By Corollary \ref{cor: rr0hereditary}, the inclusion $\theta(\pi(C))\subset \overline{(x+J)B/J(x+J)}$ has real rank zero, so there exists a projection $q\in \overline{(x+J)B/J(x+J)}$ such that 
\begin{equation}\label{eq: PermProp3}
\|(x+J)-q(x+J)\|\leq\epsilon.
\end{equation}
By assumption, $q$ lifts to a projection $p$ in $B$, which can be assumed to be in $\overline{xBx}$ by {\cite[Lemma 3.13]{rr0}}.

Since $J$ has real rank zero, we get that $(1-p)\overline{xJx}(1-p)$ has real rank zero.\ Moreover, \eqref{eq: PermProp3} gives that $\|(x-px)+J\|\leq\epsilon$, so there exists a projection $r\in(1-p)\overline{xJx}(1-p)$ such that 
\begin{equation}\label{eq: PermProp4}
\|(x-px)-r(x-px)\|\leq2\epsilon.
\end{equation}
Note that $p$ and $r$ are orthogonal, so let $p_1=p+r$ be a projection in $\overline{xB^{\dagger}x}$. Since $xy=0$, it follows that $p_1y=0$. Then, by \eqref{eq: PermProp4}, we also have that $\|x-p_1x\|\leq2\epsilon$, so $A\subseteq B$ has real rank zero by Theorem \ref{thm: eqdefnonunital}(iv).
\end{proof}

\section{K-theoretic properties}\label{sect: Kth}

In this section we characterise inclusions of real rank zero by appealing to the pairing between $K$-theory and traces.\ The motivation for this approach comes from \cite{rordamZ}, where Rørdam proved that for a certain class of $\C$-algebras, real rank zero is equivalent to the image of the pairing map in $K$-theory to be uniformly dense in the set of affine functions on the trace space (\cite[Theorem 7.2]{rordamZ}).\ The proof of the next result essentially follows the strategy in {\cite[Proposition 7.1]{rordamZ}} and {\cite[Theorem 7.2]{rordamUHF}}.

Suppose that $A\subseteq B$ is a unital inclusion of $\C$-algebras.\ Let $T(A)$ denote the space of tracial states on $A$ and $\Aff(T(A))$ the space of continuous affine functions $T(A)\to\mathbb{R}$.\ Then there is a canonical map $\gamma:T(B)\to T(A)$ induced by restriction.\ Moreover, there is a canonical affine map $\gamma^*:\Aff(T(A))\to \Aff(T(B))$, which is dual to $\gamma$ (see \cite{Kad51}).\ Moreover, for a unital $\C$-algebra $C$, the natural pairing between $K$-theory and traces is a map $\rho_C:K_0(C)\to \Aff(T(C))$.\ For $n\in\mathbb{N}$ and $\tau\in T(C)$, write $\tau_n$ for the canonical non-normalized extension of $\tau$ to a tracial functional on $M_n(C)$ and define

$$\rho_C([p]_0 - [q]_0)(\tau) := \tau_n(p - q)$$
for all $\tau\in T(C)$ and all projections $p,q\in M_n(C)$.

If $C$ is a unital $\C$-algebra and $a,b\in C\otimes \mathcal{K}$ are positive, we will write $a\preceq b$ to mean that there exists a sequence $(x_n)_{n\in\mathbb{N}}\subset C\otimes \mathcal{K}$ such that $x_nbx_n^*\to a$ as $n\to\infty$.\ In this case, $a$ is said to be \emph{Cuntz subequivalent} to $b$. Given $\tau\in T(C)$ and $a \in M_n(C)_+$, define $$d_{\tau}(a) :=\lim\limits_{r\to\infty} \tau_n(a^{1/r}),$$
where $\tau_n$ is the non-normalized extension of $\tau$ to $M_n(C)$.\ Following \cite[Definition 1.5]{BBSTWW}, we say that $C$ has \emph{strict comparison of positive elements with
respect to tracial states} if whenever $a, b \in M_n(C)_+$ for some $n\in\mathbb{N}$ and $a\in\overline{M_n(C)bM_n(C)}$, we have 
$$d_{\tau}(a)<d_{\tau}(b) \ \text{for all} \ \tau\in T(C)\ \implies a\preceq b.$$ Moreover, $C$ is said to have \emph{stable rank one} if invertible elements are dense in $C$ (see \cite{SR1}).

\begin{theorem}\label{thm: traceimpliesRR0}
Let $A\subseteq B$ be a unital inclusion of separable $\C$-algebras such that every nonzero positive element in $A$ is full in $B$.\ Suppose further that $B$ has stable rank one and strict comparison of positive elements with respect to tracial states.\ If $\gamma^*\left(\Aff(T(A))\right)\subseteq\overline{\rho_B(K_0(B))}$, then the inclusion $A\otimes\mathcal{K}\subseteq B\otimes\mathcal{K}$ has real rank zero.
\end{theorem}

\begin{proof}
We will first show that the inclusion $A\subseteq B$ has real rank zero.\ Let $a\in A$ be a nonzero positive element and $E=\overline{aBa}$.\ Then, for any $\epsilon>0$ define $f_{\epsilon}:\mathbb{R}_+\to\mathbb{R}$ by  \[  f_{\epsilon}(t)= 
\begin{cases} 
      0 & t\leq \epsilon \\
      \frac{t-\epsilon}{\epsilon} & \epsilon\leq t\leq 2\epsilon \\
      1 & 2\epsilon\leq t.
   \end{cases}
\]
We claim that $E$ has an approximate unit of projections.\ For this, it suffices to find projections $p_j\in B$ such that $$f_{\delta_1}(a)\leq p_1\leq f_{\delta_2}(a)\leq p_2\leq \ldots,$$ where $\delta_j=16^{-j}$. Then $\{p_j\}_{j\geq 1}$ will be an approximate unit for $E$. For each $\delta>0$, it suffices to find a projection $p\in B$ such that $$f_{\delta}(a)\leq p \leq f_{\delta/16}(a).$$

Denote by $\sigma(a)$ the spectrum of $a$ in $B$.\ If $\sigma(a)\cap (\delta/8,\delta/2)=\emptyset$, then we can take the projection $p=f(a)$, where 
\begin{equation}\label{eq: intermf}
f(t)= 
\begin{cases} 
      0 & t\leq \delta/8 \\
      \text{linear} & \delta/8\leq t\leq \delta/2 \\
      1 & \delta/2\leq t.
   \end{cases}
\end{equation} If there exists $\lambda\in \sigma(a)\cap (\delta/8,\delta/2)$, then take $g$ to be a continuous function, pointwise smaller than $f_{\delta/8}$, with open support $(\delta/8,\delta/2)$.\ One can note that $f_{\delta/8}(a)\geq g(a)+f_{\delta/2}(a)$ with $g$ and $f_{\delta/2}$ orthogonal.\ Therefore,
\begin{equation}
    d_{\tau}(f_{\delta/2}(a)) + d_{\tau}(g(a))\leq d_{\tau}(f_{\delta/8}(a)),\label{eq4}
\end{equation} for all $\tau\in T(B)$. Since $\sigma(a)\cap (\delta/8,\delta/2)\neq \emptyset$, $g(a)$ is non-zero.\ By assumption, $g(a)$ is full in $B$, so
\begin{equation}
    d_{\tau}(f_{\delta/2}(a))< d_{\tau}(f_{\delta/8}(a)),\label{eq5}
\end{equation} for all $\tau\in T(B)$.

Consider the function $f$ defined in \eqref{eq: intermf}.\ Then,
\begin{equation}\label{eq: RHSIneq}
\tau(f(a))< {\tau}(f_{\delta/8}(a)) \leq d_{\tau}(f_{\delta/8}(a)),
\end{equation} for all $\tau\in T(B)$.\ Moreover, if we consider a continuous function $h$, pointwise smaller than $f$, and with open support $(\delta/8,\delta/2)$, we see that $$\tau(f_{\delta/2}(a)^{1/n})+\tau(h(a))<\tau(f(a))$$ for any $\tau\in T(B)$ and any $n\in\mathbb{N}$.\ Since $\sigma(a)\cap (\delta/8,\delta/2)\neq \emptyset$, $h(a)$ is non-zero and hence full by assumption.\ Combining this with \eqref{eq: RHSIneq} yields that $$ d_{\tau}(f_{\delta/2}(a))<\tau(f(a))< d_{\tau}(f_{\delta/8}(a))$$ for any $\tau\in T(B)$.

The map $\tau\mapsto \tau(f(a))$ is in the image of $\gamma^*$, which is contained in $\overline{\rho_B(K_0(B))}$ by assumption.\ Therefore, we can find $g\in K_0(B)$ such that
\begin{equation}
    d_{\tau}(f_{\delta/2}(a))< \tau(g)< d_{\tau}(f_{\delta/8}(a)),\label{eq7}
\end{equation} for all $\tau\in T(B)$. Since $B$ has strict comparison of positive elements with respect to tracial states, we get that $g=[q]_0\in K_0(B)$, for some projection $q\in M_n(B)$.\footnote{This is a standard trick and goes as follows: write $g= [q']_0 - [p']_0$ for projections $p',q' \in M_n(B)$ for some $n\in \mathbb N$. As $\tau(g) > 0$ for all $\tau\in T(B)$ we have $\tau_n(p') < \tau_n(q')$ for all $\tau\in T(B)$. Hence $p'$ is subequivalent to $q'$ so  $p' \sim q'' \leq q'$ for some projection $q''$. Letting $q = q' - q''$, we have $g = [q]_0 \in K_0(B)$.} Moreover, since $d_\tau(q)=\tau(q)$, strict comparison of $B$ yields that $$f_{\delta/2}(a)\preceq q\preceq f_{\delta/8}(a).$$ 

Furthermore, $f_{\delta/8}(a)$ is full in $B$, so the hereditary $\C$-subalgebra $\overline{f_{\delta/8}(a)Bf_{\delta/8}(a)}$ is stably isomorphic to $B$ by \cite[Theorem 2.8]{Brown77}.\ Since stable rank one is preserved by stable isomorphisms (\cite[Theorem 3.6]{SR1}), it follows that $\overline{f_{\delta/8}(a)Bf_{\delta/8}(a)}$ has stable rank one.\ Then, one can follow the last part of the proof of \cite[Theorem $7.2$]{rordamUHF} to obtain a projection $p\in B$ such that $f_{\delta}(a)\leq p \leq f_{\delta/16}(a).$

Let $n\in\mathbb{N}$ and consider the unital inclusion $M_n(A)\subseteq M_n(B)$.\ Since $A\subseteq B$ is full, so is the inclusion $M_n(A)\subseteq M_n(B)$ (\cite[Lemma 3.16]{oinftyclass}).\ Moreover, strict comparison passes to matrix amplifications, and so does stable rank one (\cite[Theorem 3.3]{SR1}).\ Lastly, one also has that $\gamma^*\left(\Aff(T(M_n(A)))\right)\subseteq\overline{\rho_{M_n(B)}(K_0(M_n(B)))}$.\ Hence the proof above shows that the inclusion $M_n(A)\subseteq M_n(B)$ has real rank zero.\ Thus, $A\otimes\mathcal{K}\subseteq B\otimes\mathcal{K}$ has real rank zero by Proposition \ref{prop: rr0inductivelim}.
\end{proof}

The next theorem provides a converse for Theorem \ref{thm: traceimpliesRR0}.

\begin{theorem}\label{thm: RR0impliestrace}
Let $A\subseteq B$ be a unital inclusion of separable $\C$-algebras such that $A$ is exact, has stable rank one and no nonzero finite dimensional representations.\ If $A\otimes\mathcal{K}\subseteq B\otimes\mathcal{K}$ has real rank zero, then $\gamma^*\left(\Aff(T(A))\right)\subseteq\overline{\rho_B(K_0(B))}$. 
\end{theorem}

\begin{proof}
If $T(B)$ is the empty set, then the conclusion is vacuously true, so let us assume that $T(B)\neq \emptyset$.\ Let $f$ be a strictly positive continuous affine function on $T(A)$.\ Moreover, $A$ is unital and exact, so all quasitraces are traces by \cite{haagerup}.\ Then, there exists a positive element $a\in A\otimes\mathcal{K}$ such that $d_{\tau}(a)=f(\tau)$ for all $\tau\in T(A)$ (\cite[Theorem 7.14]{CuntzSemigrRanks} see also the second theorem in \cite[pp.38]{CuntzSemigrRanks}).\ Since the inclusion $A\otimes\mathcal{K}\subseteq B\otimes\mathcal{K}$ has real rank zero, there exists an increasing approximate unit of projections $(p_n)_{n\geq 1}$ for the hereditary subalgebra generated by $a\in A\otimes\mathcal{K}$ in $B\otimes\mathcal{K}$ (\cite[Proposition 1.2]{Zha91}).

In particular, for any $\tau\in T(B)$, $\tau(p_n)=d_{\tau}(p_n)$ converges pointwise to $d_{\tau}(a)=f(\tau)$, where $f(\tau)$ denotes the evaluation of $f$ at the restriction of $\tau$ to $A$. But since $f$ is continuous, Dini's theorem gives that $\tau(p_n)$ converges uniformly to $f(\tau)$. Thus, $\gamma^*(f)$ lies in $\overline{\rho_B(K_0(B))}$.

Now, if $f$ is any continuous affine function on $T(A)$, there exists some $n\in\mathbb{N}$ such that $f+n$ is strictly positive. Since $\gamma^*(n)$ is the constant function $n$ and the inclusion is unital, $\gamma^*(n)\in\overline{\rho_B(K_0(B))}$.\ Hence, $\gamma^*(f)$ lies in $\overline{\rho_B(K_0(B))}$. 
\end{proof}

\begin{remark}
We do not know if real rank zero passes to matrix amplifications, hence the assumption that the stabilised inclusion has real rank zero in the theorem above.
\end{remark}

\begin{question}
Let $A\subseteq B$ be an inclusion of $\C$-algebras and suppose that $A\subseteq B$ has real rank zero.\ Does the induced inclusion $M_2(A)\subseteq M_2(B)$ have real rank zero?
\end{question}

In fact, further assuming strict comparison on $B$ we can strengthen the result in Theorem \ref{thm: RR0impliestrace}.\ Similarly to {\cite[Theorem 7.2]{rordamZ}}, we can show that projections coming from $A\otimes\mathcal{K}$ can be weakly divided in $K_0(B)_+$. 

\begin{theorem}\label{weakdiv}
Let $A\subseteq B$ be a unital inclusion of separable $\C$-algebras such that $A$ is exact, has stable rank one and no nonzero finite dimensional representations.\ Suppose further that $B$ is simple, stably finite, and has strict comparison of projections with respect to tracial states.\ If $A\otimes\mathcal{K}\subseteq B\otimes\mathcal{K}$ has real rank zero, then for all $x\in K_0(A)_+$ and all $n\geq 2$, there exist $g_x^{(n)},h_x^{(n)}\in K_0(B)_+$ such that $$K_0(\iota)(x)=ng_x^{(n)}+(n+1)h_x^{(n)},$$ where $\iota$ denotes the inclusion map of $A$ into $B$.
\end{theorem}

\begin{proof}
Let $x\in K_0(A)_+$ and $n\geq 2$.\ If $K_0(\iota)(x)=0$, then we are done, so let us suppose that $K_0(\iota)(x)\in K_0(B)_+$ is nonzero.\ Since $B$ is simple, stably finite and $K_0(\iota)(x)\in K_0(B)$ is positive and nonzero, it follows that $K_0(\iota)(x)$ is an order unit (see \cite[Corollary 6.3.6]{blackadar}).\ Moreover, $T(B)$ is compact by unitality of $B$, so we get that $\inf\limits_{\tau\in T(B)}\tau(K_0(\iota)(x))>0$.\ Let $\epsilon<\frac{1}{2n(n+1)}\inf\limits_{\tau\in T(B)}\tau(K_0(\iota)(x))$.

Recall that $\rho_A(x):T(A)\to\mathbb{R}$ is the affine continuous function given by $\rho_A(x)(\tau)=\tau(x)$ for any $\tau\in T(A)$.\ Then $\rho_A(x)\in \Aff(T(A))$ and since this is a vector space, $\alpha\rho_A(x)\in\Aff(T(A))$ for any $\alpha\in\mathbb{R}$.\ Take $\alpha_n=\frac{\frac{1}{n}+\frac{1}{n+1}}{2}$ and note that $\gamma^*(\alpha_n\rho_A(x))(\tau)=\alpha_n\tau(K_0(\iota)(x))$ for any $\tau\in T(B)$.\ Then, by Theorem \ref{thm: RR0impliestrace}, the image of the canonical map $\gamma^*:\Aff(T(A))\to \Aff(T(B))$ induced by $\iota$ is contained in $\overline{\rho_B(K_0(B))}$.\ Hence, there exists $y_n\in K_0(B)$ such that $$|\tau(y_n)-\alpha_n\tau(K_0(\iota)(x))|<\epsilon$$ for any $\tau\in T(B)$. Therefore, 
\begin{align*}
n\tau(y_n)&<n\alpha_n\tau(K_0(\iota)(x))+n\epsilon \\
&\leq n\Big(\alpha_n+\frac{1}{2n(n+1)}\Big)\tau(K_0(\iota)(x)) \\
&=\tau(K_0(\iota)(x))
\end{align*}for any $\tau\in T(B)$.\ Moreover, by the choice of $\epsilon$ and $\alpha_n$, we also get that $\tau(y_n)>0$ for any $\tau\in T(B)$.\ But $B$ has strict comparison, so $0\leq ny_n\leq K_0(\iota)(x)$.\ A similar calculation using that $(n+1)\tau(y_n)> (n+1)\alpha_n\tau(K_0(\iota)(x))-(n+1)\epsilon$ and strict comparison of $B$ yields that $(n+1)y_n\geq K_0(\iota)(x)$. 

Finally, set $g_x^{(n)}=(n+1)y_n-K_0(\iota)(x)$ and $h_x^{(n)}=K_0(\iota)(x)-ny_n$. It is immediate to check that $K_0(\iota)(x)=ng_x^{(n)}+(n+1)h_x^{(n)}$.
\end{proof}

\section{Purely infinite inclusions}\label{sect: purelyinf}

Cuntz introduced pure infiniteness of simple $\C$-algebras in \cite{Cuntz-Annals} as a $\C$-algebraic analogue of type $\mathrm{III}$ von Neumann algebra factors which have good $K$-theoretic properties. Following Kirchberg and Rørdam in \cite{purelyinfKR}, a (not necessarily simple) $\C$-algebra is \emph{purely infinite} if every non-zero positive element $a\in A$ is properly infinite, in the sense that $a\oplus a$ is Cuntz subequivalent to $a$. In this section we consider inclusions $A\subseteq B$ where each non-zero positive element in $A$ is properly infinite in $B$.

A classical result of Zhang (\cite{zhangRR0}) shows that any simple purely infinite $\C$-algebra has real rank zero.\ We will now prove the corresponding result for inclusions.

\begin{theorem}\label{thm: purelyinfrr0}
Let $A\subseteq B$ be an inclusion of $\C$-algebras such that every nonzero positive element in $A$ is full and properly infinite in $B$.\ Then $A\subseteq B$ has real rank zero.\ Moreover, for any nonzero positive element $a\in A$, $\overline{aBa}$ has an increasing approximate unit of full, properly infinite projections.
\end{theorem}

\begin{proof}
Let $a\in A_+$ be nonzero and $\epsilon>0$.\ It suffices to show that there exists a projection $p\in \overline{aBa}$ such that $p(a-\epsilon)_+=(a-\epsilon)_+$.\ We will construct this projection by obtaining a suitable scaling element in the sense of \cite[Definition 1.1]{ScalingElem} (see also \cite[Remark 2.4]{pasnicurordam}).\footnote{A contraction $x$ is called a scaling element if $x^*xxx^*=xx^*$.} Following  {\cite[Remark 2.4]{pasnicurordam}}, to obtain such a projection $p$, it suffices to find a contraction $x\in \overline{aBa}$ such that $$x^*xxx^*=xx^*, \quad x^*x(a-\epsilon)_+=(a-\epsilon)_+, \quad xx^*(a-\epsilon)_+=0.$$

Consider the spectrum of $a$ in $B$ denoted by $\sigma(a)$. If $\sigma(a)\cap (\epsilon/2,\epsilon)=\emptyset$, then we can take $p$
to be the orthogonal complement of the spectral projection of $a$ supported on $(\epsilon/2,\epsilon)$.\ Else, let $f:[0,\infty)\to[0,1]$ be a continuous function with open support $(\epsilon/2,\epsilon)$.\ Then, $f(a)$ is full and properly infinite in $B$, so $a\preceq f(a)$ ({\cite[Proposition 3.5]{purelyinfKR}}), with the Cuntz subequivalence taking place in the hereditary subalgebra $\overline{aBa}$.

Applying {\cite[Remark 2.5]{pasnicurordam}} to the subequivalence $a\preceq f(a)$, we can find a contraction $x\in \overline{aBa}$  such that \begin{equation}\label{scaling}
    x^*x(a-\epsilon/2)_+=(a-\epsilon/2)_+
\end{equation} and $$xx^*\in\overline{f(a)Bf(a)}\subseteq \overline{(a-\epsilon/2)_+B(a-\epsilon/2)_+}.$$

It follows that $(x^*x)(xx^*)=xx^*$, so $x$ is a scaling element in the sense of \cite[Definition 1.1]{ScalingElem}.\ Moreover, as $(a-\epsilon)_+\in \overline{(a-\epsilon/2)_+B(a-\epsilon/2)_+}$,  $x^*x(a-\epsilon)_+=(a-\epsilon)_+$.\ Furthermore, $xx^*(a-\epsilon)_+=0$, as $f(a)$ and $(a-\epsilon)_+$ are orthogonal. Following {\cite[Remark 2.4]{pasnicurordam}}, $v=x+(1-x^*x)^{1/2}$ is an isometry in the unitisation of $\overline{aBa}$ such that $p=1-vv^*$ is a projection in $\overline{aBa}$ with $p(a-\epsilon)_+=(a-\epsilon)_+$. Hence $A\subseteq B$ has real rank zero.

Note that the projection $p$ found above is full and properly infinite in $B$ since $(a-\epsilon)_+\preceq p$.\ Moreover, applying the proof above for every $n\in\mathbb{N}$, one can find a sequence of properly infinite projections $p_n\in \overline{\left(a-\frac{1}{n+1}\right)_+B\left(a-\frac{1}{n+1}\right)_+}$ such that $p_n\left(a-\frac{1}{n}\right)_+=\left(a-\frac{1}{n}\right)_+$.\ Hence, $(p_n)_{n\geq k}$, for any $k$ such that $p_k\neq 0$, is an increasing approximate unit for $\overline{aBa}$ consisting of full, properly infinite projections.
\end{proof}

A particular, but important class of inclusions as in the statement of Theorem \ref{thm: purelyinfrr0} is provided by the $^*$-homomorphisms classified by the first named author in \cite{oinftyclass}.\ Therefore, we aim to obtain an equivalent characterisation for inclusions covered by Theorem \ref{thm: purelyinfrr0}.

\begin{remark}
From the proof of Theorem \ref{thm: purelyinfrr0}, we see that whenever $\overline{aBa}$ is not unital i.e. $0\in\sigma(a)$ is not isolated, then the approximate unit of projections $(p_n)_n$ in $\overline{aBa}$ may be chosen to be increasing, and such that each $p_n$ is full, properly infinite and with $[p_n]_0=0$ in $K_0(B)$.

In fact, the projection $p$ in the proof is of the form  $p=1-vv^\ast$ for an isometry $v$ in the unitisation of $\overline{aBa}$, and hence $[p]_0=0$ in $K_0(B)$.
\end{remark}

\begin{cor}\label{cor: equivalencepurelyinf}
Let $A\subseteq B$ be an inclusion of $\C$-algebras.\ Then the following are equivalent:
\begin{enumerate}
    \item every nonzero positive element in $A$ is full and properly infinite in $B$;
    \item for any nonzero positive element $a\in A$, $\overline{aBa}$ has a full properly infinite projection.
\end{enumerate}
\end{cor}

\begin{proof}
Theorem \ref{thm: purelyinfrr0} shows that (i) implies (ii), so it remains to show that (ii) implies (i).\ Let $a\in A$ be a nonzero positive element and let $p\in\overline{aBa}$ be a full properly infinite projection.\ Then, $a$ is in the ideal generated by $p$, so $a\preceq p$ ({\cite[Proposition 3.5]{purelyinfKR}}).\ Since $p\in\overline{aBa}$, it also follows that $p\preceq a$ ({\cite[Proposition 2.7 (i)]{purelyinfKR}}), so $a$ is full. Moreover, $$a\oplus a\preceq p\oplus p\preceq p\preceq a,$$ which implies that $a$ is properly infinite.
\end{proof}

\begin{remark}
Inclusions with infinite behaviour of this type have been essential in proving that various natural constructions are purely infinite.\ For example, in {\cite[Theorem 3.3]{purelyinfcrossedproducts}}, Rørdam and Sierakowski show that, under some assumptions, the crossed product $A\rtimes_r G$ is purely infinite if and only if any nonzero positive element in $A$ is properly infinite in $A\rtimes_r G$.\ Similarly, in \cite{purelyinfgroupoids}, Bönicke and Li study when a groupoid $\C$-algebra is purely infinite.\ If $\mathcal{G}$ is a groupoid with unit space $\mathcal{G}^{(0)}$, they use the canonical inclusion $C_0(\mathcal{G}^{(0)})\subseteq \C_r(\mathcal{G})$ to describe when $\C_r(\mathcal{G})$ is purely infinite (see \cite[Theorem B]{purelyinfgroupoids}).
\end{remark}

Assuming nuclearity and $\mathcal{O}_\infty$-stability, we can strenghten the result in Theorem \ref{thm: purelyinfrr0} and show that $^*$-homomorphisms classified by the first named author in \cite{oinftyclass} are approximately factoring through a $\C$-algebra of real rank zero. In many cases of interest, this algebra can be assumed to be a Kirchberg algebra. 

First, let us record some definitions. Let $A$ and $B$ be $\C$-algebras with $A$ separable and let $\theta: A \to B$ be a $^*$-homomorphism.\ Then, following {\cite[Definition 3.16]{O2stableclassif}}, $\theta$ is said to be \emph{$\mathcal{O}_\infty$-stable} if $\mathcal{O}_\infty$ embeds unitally in $B_\infty \cap \theta(A)'/\Ann(\theta(A))$.\footnote{Note that $\Ann(\theta(A))=\{b\in B_\infty: b\theta(A)+\theta(A)b=\{0\}\}$.} Two $^*$-homomorphisms $\phi,\psi:A\to B$, with $A$ separable, are \emph{approximately Murray-von Neumann equivalent} if  there exists $v \in B_\infty$ such that $v^*\phi(a)v = \psi(a)$ and $v\psi(a)v^* = \phi(a)$ for all $a \in A$.\ 

In \cite{oinftyclass}, full, nuclear, $\mathcal{O}_\infty$-stable $^*$-homomorphisms are shown to be classified up to approximate Murray-von Neumann equivalence by KL$_{\nuc}$ (see \cite{Skandalis} or \cite[Section 8.1]{oinftyclass} for details on KL$_{\nuc}$ and KK$_{\nuc}$).\ Moreover, if $\theta:A\to B$ is a full $\mathcal{O}_\infty$-stable $^*$-homomorphism, then any nonzero and positive element $\theta(a)$ is properly infinite in $B$ (\cite[Lemma 4.4]{BGSW22}).\ Thus, the inclusion $\theta(A)\subseteq B$ has real rank zero by Theorem \ref{thm: purelyinfrr0}.\ In fact, under some mild additional assumption, we can show that $\theta$ is equivalent to a $^*$-homomorphism which factors through a $\C$-algebra of real rank zero.

\begin{theorem}\label{thm: factKirch}
Let $A$ be a separable exact $\C$-algebra, $B$ be a $\sigma$-unital $\C$-algebra and $\theta:A\to B$ be a full, nuclear, $\mathcal{O}_\infty$-stable $^*$-homomorphism.\ Then $\theta$ is approximately Murray-von Neumann equivalent to a $^*$-homo-morphism which factors through a separable exact $\C$-algebra $C$ with real rank zero.\ Moreover, if $A$ is $\KK$-equivalent to a nuclear $\C$-algebra,\footnote{For example if $A$ is nuclear or if it satisfies the $\UCT$.
} then $C$ can be taken to be a Kirchberg algebra.\footnote{A simple, separable, nuclear, purely infinite $\C$-algebra is called a Kirchberg algebra.}
\end{theorem}

\begin{proof}
Building on the work of Pimsner in \cite{pimsner}, Kumjian showed in \cite{KumjianKKeqInclusion} that there exists a simple separable purely infinite $\C$-algebra $C$ and an inclusion $\iota:A\to C$ such that $\iota$ is a $\KK$-equivalence.\ The existence of such $C$ is proved in \cite[Theorem 3.1]{KumjianKKeqInclusion} (note that nuclearity is only needed to obtain nuclearity of $C$), while the fact that $\iota$ induces a $\KK$-equivalence follows from {\cite[Corollary 4.5]{pimsner}}.\ Moreover, since $A$ is exact, $C$ can also be chosen to be exact by {\cite[Corollary 3.6]{exactnessCuntzPimsner}}. 

Since the product of a $\KK$-element with a $\KK_{\nuc}$-element is in $\KK_{\nuc}$ (\cite[Proposition 2.2(b)]{Skandalis}), the product $\KK_{\nuc}(\theta)\circ \KK(\iota)^{-1}$ is an element of $\KK_{\nuc}(C,B)$.\ Then, there exists a full, nuclear, $\mathcal{O}_{\infty}$-stable $^*$-homomorphism $\psi:C\to B$ such that $\KK_{\nuc}(\psi)=\KK_{\nuc}(\theta)\circ \KK(\iota)^{-1}$({\cite[Theorem $A$]{oinftyclass}}).\ Now we apply the uniqueness part of the classification of full, nuclear, $\mathcal{O}_{\infty}$-stable $^*$-homomorphisms.\ Precisely, since $\KK_{\nuc}(\psi\circ\iota)=\KK_{\nuc}(\theta)$, $\theta$ and $\psi\circ\iota$ are approximately Murray-von Neumann equivalent ({\cite[Theorem 8.10]{oinftyclass}}). The fact that $C$ has real rank zero follows from \cite{zhangRR0} (alternatively apply Theorem \ref{thm: purelyinfrr0} to the inclusion $C\subseteq C$).

Suppose now that $A$ is $\KK$-equivalent to a separable nuclear $\C$-algebra $D$. Let $x\in \KK(A,D)$ be an element inducing the equivalence. Note that since $D$ is nuclear, $\KK_{\nuc}(A,D)=\KK(A,D)$.\ Applying {\cite[Theorem 3.1]{KumjianKKeqInclusion}} to the $\C$-algebra $D$, there exists a simple separable purely infinite $\C$-algebra $C$ and an inclusion $\iota:D\to C$ such that $\iota$ is a $\KK$-equivalence. Since $D$ is nuclear, it follows from {\cite[Theorem 3.1]{KumjianKKeqInclusion}} that we can further assume that $C$ is nuclear, so $C$ is a Kirchberg algebra.

Then, there exists a full, nuclear, $\mathcal{O}_{\infty}$-stable $^*$-homomorphism $\phi:A\to C$ such that $$\KK_{\nuc}(\phi)=\KK_{\nuc}(\iota)\circ x$$ ({\cite[Theorem $A$]{oinftyclass}}).\ Applying existence again, there exists a full, nuclear, $\mathcal{O}_{\infty}$-stable $^*$-homomorphism $\psi:C\to B$ such that $$\KK_{\nuc}(\psi)=\KK_{\nuc}(\theta)\circ x^{-1}\circ \KK_{\nuc}(\iota)^{-1}.$$  Since $\KK_{\nuc}(\psi\circ\phi)=\KK_{\nuc}(\theta)$, $\theta$ and $\psi\circ\phi$ are approximately Murray-von Neumann equivalent by {\cite[Theorem 8.10]{oinftyclass}}.\ Hence, $\theta$ is approximately Murray-von Neumann equivalent to a $^*$-homomorphism which factors through a Kirchberg algebra.
\end{proof}

Let us now examine some particular inclusions arising from dynamics that can be shown to have real rank zero.\ In fact, they factor exactly through $\C$-algebras of real rank zero.

\begin{theorem}\label{roealgebraex}
Let $G$ be a countable, discrete, exact group.\ If we consider the action of $G$ on $\ell^\infty(G)$ by left translation, then the canonical inclusion $\C_r(G)\subseteq\ell^{\infty}(G)\rtimes_r G$ has real rank zero.\ In fact, it factors through a $\C$-algebra of real rank zero. 
\end{theorem}

\begin{proof}
Let us denote this inclusion by $\iota$. If the group is non-amenable, there is an action $\alpha$ of $G$ on the Cantor space $X$ such that $C(X)\rtimes_{\alpha,r} G$ is a Kirchberg algebra ({\cite[Theorem 6.11]{purelyinfcrossedproducts}}). Since the $\C$-algebra $C(X)\rtimes_{\alpha,r} G$ is simple and purely infinite, it has real rank zero by \cite{zhangRR0}. 

Suppose now that the group $G$ is amenable. Then, combining \cite[Theorem 5.4]{Suzuki} and \cite[Lemma 5.2]{Suzuki}, there exists a totally disconnected compact Hausdorff space $X$ and a free minimal action $\alpha$ of $G$ on $X$ such that the transformation groupoid $X\rtimes_\alpha G$ is almost finite.\footnote{All groupoids in \cite{Suzuki} are assumed to be Hausdorff.} Since the groupoid is also minimal and has totally disconnected unit space, the crossed product $C(X)\rtimes_\alpha G$ has real rank zero by \cite[Theorem 3.2]{ABBL20}.\footnote{Note that in \cite{ABBL20}, minimal almost finite groupoids are assumed to be ample with compact unit space. In this case, ampleness is ensured by the fact that $X$ is totally disconnected and compact. Moreover, the definitions of almost finiteness in \cite{ABBL20} and \cite{Suzuki} coincide, as observed in the comments following \cite[Remark 2.1]{ABBL20}.}

We will now combine the two cases above. Let $X$ be a compact Hausdorff space and an action $\alpha$ of $G$ on $X$ such that the crossed product $C(X)\rtimes_\alpha G$ has real rank zero. Fix $x\in X$ and let $C(X)\to \ell^{\infty}(G)$ be the equivariant $^*$-homomorphism given by $f\mapsto \left(f(\alpha_g(x))\right)_{g\in G}$.\ This induces a unital $^*$-homomorphism $C(X) \rtimes_{\alpha,r} G \to \ell^{\infty}(G) \rtimes_r G$. 

It is readily seen that in both cases, the composition
\begin{equation}
\C_r(G) \to C(X) \rtimes_{\alpha,r} G \to \ell^{\infty}(G) \rtimes_r G
\end{equation}
is the canonical inclusion $\iota$. Hence, the inclusion $\C_r(G)\subseteq\ell^{\infty}(G)\rtimes_r G$ factors through a $\C$-algebra of real rank zero. 
\end{proof}

\begin{remark}
Note that the inclusion $\C_r(G)\subseteq\ell^{\infty}(G)\rtimes_r G$ has real rank zero even if neither $\C_r(G)$ nor $\ell^{\infty}(G)\rtimes_r G$ have real rank zero. For example, take $G=\mathbb{F}_2\times \mathbb{Z}^2$ and note that its uniform Roe algebra $\ell^{\infty}(G)\rtimes G$ does not have real rank zero by {\cite[Corollary 3.5]{RRRoeAlg}}. Moreover, $\C_r(G)\cong \C_r(\mathbb{F}_2)\otimes C(\mathbb{T}^2)$, which does not have real rank zero either, since its primitive ideal space is $\mathbb{T}^2$ (by simplicity of $\C_r(\mathbb F_2)$ shown in \cite{Powers}), which is not totally disconnected.
\end{remark}

Another natural inclusion from dynamics is realised by embedding the reduced group $\C$-algebra into the reduced crossed product of $C(\partial_F G)$ by $G$. Here $\partial_F G$ denotes the Furstenberg boundary of $G$ and we refer the reader to \cite{furstenbergboundary} for a detailed description. 

\begin{theorem}\label{furstbdex}
Let $G$ be a countable, discrete, non-elementary hyperbolic group with no nontrivial finite normal subgroups. Then the canonical inclusion $\C_r(G)\subseteq C(\partial_F G) \rtimes_r G$ has real rank zero.\ In fact, it factors through a $\C$-algebra of real rank zero.     
\end{theorem}

\begin{proof}
If we denote the Gromov boundary of $G$ by $\partial G$, then there is an equivariant inclusion $C(\partial G)\subseteq C(\partial_F G)$ ({\cite[Remark 5.6]{furstenbergboundary}}). Therefore, there is an induced unital $^*$-homomorphism $C(\partial G)\rtimes_r G\to C(\partial_F G)\rtimes_r G$. 
It is readily seen that the composition
\begin{equation}
\C_r(G) \to C(\partial G) \rtimes_r G \to C(\partial_F G) \rtimes_r G
\end{equation}
is the canonical inclusion $\C_r(G)\subseteq C(\partial_F G) \rtimes_r G$. 

We will now show that the crossed product $C(\partial G) \rtimes_r G$ is simple and purely infinite. Since the group $G$ is hyperbolic, the elements of $G$ are elliptic, parabolic, or hyperbolic (\cite[Section 8.3]{HypGroups}). Then, as $G$ is non-elementary and has no nontrivial finite normal subgroups, all elements of $G$ are hyperbolic (\cite[Proposition 8.28, Theorem 8.29]{HypGroups}).\footnote{An element $g\in G$ is called hyperbolic if for any $x\in \partial G$, the orbit $n\in\mathbb Z \mapsto g^nx$ is a quasi-geodesic (see \cite[Proposition 8.21]{HypGroups}).} Therefore, any element of $G$ fixes exactly two points of $\partial G$ (\cite[Corollary 8.20]{HypGroups}). Hence, the interior of fixed points is empty, so the action of $G$ on $\partial G$ is topologically free (\cite[Definition 4]{LacaSpielberg}). Moreover, by \cite[Example 2.1]{LacaSpielberg}, the action of $G$ on $\partial G$ is a strong boundary action (in the sense of \cite[Definition 1]{LacaSpielberg}). Hence, $C(\partial G) \rtimes_r G$ is simple and purely infinite by \cite[Theorem 5]{LacaSpielberg}. In particular, $C(\partial G) \rtimes_r G$ has real rank zero by \cite{zhangRR0}. Thus, the conclusion follows.
\end{proof}


\begin{remark}
We note what seems to be a persistent issue in the literature. In \cite[Proposition 3.2]{AnDel}, it is claimed that $C(\partial G) \rtimes_r G$ is simple and purely infinite if $G$ is a non-elementary hyperbolic group, while in \cite{LacaSpielberg}, this is claimed to be the case for any word hyperbolic group. If $C(\partial G) \rtimes_r G$ is simple, then $G$ is a $\C$-simple group (\cite[Theorem 1.5]{furstenbergboundary}). However, as it was pointed out to us by Julian Kranz, there are non-elementary hyperbolic groups with finite normal subgroups (e.g. $\mathbb{F}_2\times\mathbb{Z}_2)$ and in these cases, the group $G$ cannot be $\C$-simple (see for example \cite[Theorem 8.14]{DGO17}). Thus, the assumption that $G$ has no nontrivial finite normal subgroups is necessary to conclude that the crossed product $C(\partial G) \rtimes_r G$ is simple.
\end{remark}

\begin{remark}
It seems plausible to the authors that the inclusion $\C_r(G)\subseteq C(\partial_F G) \rtimes_r G$ has real rank zero for any countable, discrete, non-amenable group $G$.
\end{remark}

\bibliography{rr0}

\begin{thebibliography}{10}

\bibitem{AnDel}
C.~Anantharaman-Delaroche.
\newblock Purely infinite {$\C$}-algebras arising from dynamical systems.
\newblock {\em Bull. Soc. Math. France}, 125(2):199--225, 1997.

\bibitem{CuntzSemigrRanks}
R.~Antoine, F.~Perera, L.~Robert, and H.~Thiel.
\newblock {$\C$}-algebras of stable rank one and their {C}untz semigroups.
\newblock {\em Duke Math. J.}, 171(1):33--99, 2022.

\bibitem{ABBL20}
P.~Ara, C.~B\"{o}nicke, J.~Bosa, and K.~Li.
\newblock Strict comparison for {$\C$}-algebras arising from almost finite groupoids.
\newblock {\em Banach J. Math. Anal.}, 14(4):1692--1710, 2020.

\bibitem{blackadar}
B.~Blackadar.
\newblock {\em {$K$}-theory for operator algebras}, volume~5 of {\em Mathematical Sciences Research Institute Publications}.
\newblock Cambridge University Press, Cambridge, second edition, 1998.

\bibitem{ScalingElem}
B.~Blackadar and J.~Cuntz.
\newblock The structure of stable algebraically simple {$\C$}-algebras.
\newblock {\em Amer. J. Math.}, 104(4):813--822, 1982.

\bibitem{IrrRotRR0}
B.~Blackadar, A.~Kumjian, and M.~R{\o}rdam.
\newblock Approximately central matrix units and the structure of noncommutative tori.
\newblock {\em $K$-Theory}, 6(3):267--284, 1992.

\bibitem{purelyinfgroupoids}
C.~B\"{o}nicke and K.~Li.
\newblock Ideal structure and pure infiniteness of ample groupoid {$\C$}-algebras.
\newblock {\em Ergodic Theory Dynam. Systems}, 40(1):34--63, 2020.

\bibitem{BBSTWW}
J.~Bosa, N.~P. Brown, Y.~Sato, A.~Tikuisis, S.~White, and W.~Winter.
\newblock Covering dimension of {$\C$}-algebras and $2$-coloured classification.
\newblock {\em Mem. Amer. Math. Soc.}, 257(1233):vii+97, 2019.

\bibitem{BGSW22}
J.~Bosa, J.~Gabe, A.~Sims, and S.~White.
\newblock The nuclear dimension of {$\mathcal{O}_\infty$}-stable {$\C$}-algebras.
\newblock {\em Adv. Math.}, 401:Paper No. 108250, 51, 2022.

\bibitem{Brown77}
L.~G. Brown.
\newblock Stable isomorphism of hereditary subalgebras of {$\C$}-algebras.
\newblock {\em Pacific J. Math.}, 71(2):335--348, 1977.

\bibitem{rr0}
L.~G. Brown and G.~K. Pedersen.
\newblock {$\C$}-algebras of real rank zero.
\newblock {\em J. Funct. Anal.}, 99(1):131--149, 1991.

\bibitem{classif}
J.~R. Carri{\'o}n, J.~Gabe, C.~Schafhauser, A.~Tikuisis, and S.~White.
\newblock Classifying $^*$-homomorphisms {I}: {U}nital simple nuclear {$\C$}-algebras.
\newblock {\em arXiv:2307.06480}, 2023.

\bibitem{NucDimJorge}
J.~Castillejos.
\newblock {$\C$}-algebras and their nuclear dimension.
\newblock In {\em Mexican mathematicians in the world---trends and recent contributions}, volume 775 of {\em Contemp. Math.}, pages 41--63. Amer. Math. Soc., [Providence], RI, [2021] \copyright 2021.

\bibitem{CN23}
J.~Castillejos and R.~Neagu.
\newblock On topologically zero-dimensional morphisms.
\newblock {\em J. Funct. Anal.}, 286(9):Paper No. 110368, 35, 2024.

\bibitem{Cuntz-Annals}
J.~Cuntz.
\newblock {$K$}-theory for certain {$\C$}-algebras.
\newblock {\em Ann. of Math. (2)}, 113(1):181--197, 1981.

\bibitem{CuntzKrieger}
J.~Cuntz and W.~Krieger.
\newblock A class of {$\C$}-algebras and topological {M}arkov chains.
\newblock {\em Invent. Math.}, 56(3):251--268, 1980.

\bibitem{DGO17}
F.~Dahmani, V.~Guirardel, and D.~Osin.
\newblock Hyperbolically embedded subgroups and rotating families in groups acting on hyperbolic spaces.
\newblock {\em Mem. Amer. Math. Soc.}, 245(1156):v+152, 2017.

\bibitem{Top67}
J.~De~Groot and R.~McDowell.
\newblock Locally connected spaces and their compactifications.
\newblock {\em Illinois J. Math.}, 11(3):353--364, 1967.

\bibitem{exactnessCuntzPimsner}
K.~J. Dykema and D.~Shlyakhtenko.
\newblock Exactness of {C}untz-{P}imsner {$\C$}-algebras.
\newblock {\em Proc. Edinb. Math. Soc. (2)}, 44(2):425--444, 2001.

\bibitem{RR0Classif}
G.~A. Elliott.
\newblock On the classification of {$\C$}-algebras of real rank zero.
\newblock {\em J. Reine Angew. Math.}, 443:179--219, 1993.

\bibitem{EllEv93}
G.~A. Elliott and D.~E. Evans.
\newblock The structure of the irrational rotation {$\C$}-algebra.
\newblock {\em Ann. of Math. (2)}, 138(3):477--501, 1993.

\bibitem{FeldmanMoore}
J.~Feldman and C.~C. Moore.
\newblock Ergodic equivalence relations, cohomology, and von {N}eumann algebras. {I}.
\newblock {\em Trans. Amer. Math. Soc.}, 234(2):289--324, 1977.

\bibitem{FeldmanMooreII}
J.~Feldman and C.~C. Moore.
\newblock Ergodic equivalence relations, cohomology, and von {N}eumann algebras. {II}.
\newblock {\em Trans. Amer. Math. Soc.}, 234(2):325--359, 1977.

\bibitem{O2stableclassif}
J.~Gabe.
\newblock A new proof of {K}irchberg's {$\mathcal{O}_2$}-stable classification.
\newblock {\em J. Reine Angew. Math.}, 761:247--289, 2020.

\bibitem{oinftyclass}
J.~Gabe.
\newblock Classification of $\mathcal{O}_{\infty}$-stable {$\C$}-algebras.
\newblock {\em Mem. Amer. Math. Soc.}, 293(1461):v+115, 2024.

\bibitem{DynamicalKP}
J.~Gabe and G.~Szab\'{o}.
\newblock The dynamical {K}irchberg-{P}hillips theorem.
\newblock {\em Acta Math.}, 232(1):1--77, 2024.

\bibitem{HypGroups}
E.~Ghys and P.~de~la Harpe, editors.
\newblock {\em Sur les groupes hyperboliques d'apr\`es {M}ikhael {G}romov}, volume~83 of {\em Progress in Mathematics}.
\newblock Birkh\"{a}user Boston, Inc., Boston, MA, 1990.
\newblock Papers from the Swiss Seminar on Hyperbolic Groups held in Bern, 1988.

\bibitem{GLN23}
G.~Gong, H.~Lin, and Z.~Niu.
\newblock Homomorphisms into simple {$\mathcal{Z}$}-stable {$\C$}-algebras, {II}.
\newblock {\em J. Noncommut. Geom.}, 17(3):835--898, 2023.

\bibitem{haagerup}
U.~Haagerup.
\newblock Quasitraces on exact {$\C$}-algebras are traces.
\newblock {\em C. R. Math. Acad. Sci. Soc. R. Can.}, 36(2-3):67--92, 2014.

\bibitem{jonessubfactors}
V.~F.~R. Jones.
\newblock Index for subfactors.
\newblock {\em Invent. Math.}, 72(1):1--25, 1983.

\bibitem{Kad51}
R.~V. Kadison.
\newblock A representation theory for commutative topological algebra.
\newblock {\em Mem. Amer. Math. Soc.}, 7:39, 1951.

\bibitem{furstenbergboundary}
M.~Kalantar and M.~Kennedy.
\newblock Boundaries of reduced {$\C$}-algebras of discrete groups.
\newblock {\em J. Reine Angew. Math.}, 727:247--267, 2017.

\bibitem{GroupsRR0}
E.~Kaniuth.
\newblock Group {$\C$}-algebras of real rank zero or one.
\newblock {\em Proc. Amer. Math. Soc.}, 119(4):1347--1354, 1993.

\bibitem{purelyinfKR}
E.~Kirchberg and M.~R{\o}rdam.
\newblock Non-simple purely infinite {$\C$}-algebras.
\newblock {\em Amer. J. Math.}, 122(3):637--666, 2000.

\bibitem{KumjianKKeqInclusion}
A.~Kumjian.
\newblock On certain {C}untz-{P}imsner algebras.
\newblock {\em Pacific J. Math.}, 217(2):275--289, 2004.

\bibitem{LacaSpielberg}
M.~Laca and J.~Spielberg.
\newblock Purely infinite {$\C$}-algebras from boundary actions of discrete groups.
\newblock {\em J. Reine Angew. Math.}, 480:125--139, 1996.

\bibitem{RRRoeAlg}
K.~Li and R.~Willett.
\newblock Low-dimensional properties of uniform {R}oe algebras.
\newblock {\em J. London Math. Soc.}, 97(1):98--124, 2018.

\bibitem{LoringLifting}
T.~A. Loring.
\newblock {\em Lifting solutions to perturbing problems in {$\C$}-algebras}, volume~8 of {\em Fields Institute Monographs}.
\newblock American Mathematical Society, Providence, RI, 1997.

\bibitem{pasnicurordam}
C.~Pasnicu and M.~R{\o}rdam.
\newblock Purely infinite {$\C$}-algebras of real rank zero.
\newblock {\em J. Reine Angew. Math.}, 613:51--73, 2007.

\bibitem{pimsner}
M.~V. Pimsner.
\newblock A class of {$\C$}-algebras generalizing both {C}untz-{K}rieger algebras and crossed products by {$\mathbb{Z}$}.
\newblock In {\em Free probability theory ({W}aterloo, {ON}, 1995)}, volume~12 of {\em Fields Inst. Commun.}, pages 189--212. Amer. Math. Soc., Providence, RI, 1997.

\bibitem{Powers}
R.~T. Powers.
\newblock Simplicity of the {$\C$}-algebra associated with the free group on two generators.
\newblock {\em Duke Math. J.}, 42:151--156, 1975.

\bibitem{Renault}
J.~Renault.
\newblock Cartan subalgebras in {$\C$}-algebras.
\newblock {\em Irish Math. Soc. Bull.}, (61):29--63, 2008.

\bibitem{SR1}
M.~A. Rieffel.
\newblock Dimension and stable rank in the {$K$}-theory of {$\C$}-algebras.
\newblock {\em Proc. London Math. Soc. (3)}, 46(2):301--333, 1983.

\bibitem{rordamUHF}
M.~R{\o}rdam.
\newblock On the structure of simple {$\C$}-algebras tensored with a {UHF}-algebra. {II}.
\newblock {\em J. Funct. Anal.}, 107(2):255--269, 1992.

\bibitem{rordamZ}
M.~R{\o}rdam.
\newblock The stable and the real rank of {$\mathcal{Z}$}-absorbing {$\C$}-algebras.
\newblock {\em Internat. J. Math.}, 15(10):1065--1084, 2004.

\bibitem{RordamIrr}
M.~R{\o}rdam.
\newblock Irreducible inclusions of simple {$\C$}-algebras.
\newblock {\em Enseign. Math.}, 69(3-4):275--314, 2023.

\bibitem{purelyinfcrossedproducts}
M.~R{\o}rdam and A.~Sierakowski.
\newblock Purely infinite {$\C$}-algebras arising from crossed products.
\newblock {\em Ergodic Theory Dynam. Systems}, 32(1):273--293, 2012.

\bibitem{ChrisClassif}
C.~Schafhauser.
\newblock Subalgebras of simple {AF}-algebras.
\newblock {\em Ann. of Math. (2)}, 192(2):309--352, 2020.

\bibitem{Skandalis}
G.~Skandalis.
\newblock Une notion de nucl\'{e}arit\'{e} en {$K$}-th\'{e}orie (d'apr\`es {J}. {C}untz).
\newblock {\em $K$-Theory}, 1(6):549--573, 1988.

\bibitem{Suzuki}
Y.~Suzuki.
\newblock Almost finiteness for general \'{e}tale groupoids and its applications to stable rank of crossed products.
\newblock {\em Int. Math. Res. Not.}, (19):6007--6041, 2020.

\bibitem{TW14}
A.~Tikuisis and W.~Winter.
\newblock Decomposition rank of {$\mathcal{Z}$}-stable {$\C$}-algebras.
\newblock {\em Anal. PDE}, 7(3):673--700, 2014.

\bibitem{Watatani}
Y.~Watatani.
\newblock Index for {$\C$}-subalgebras.
\newblock {\em Mem. Amer. Math. Soc.}, 83(424):vi+117, 1990.

\bibitem{Win03}
W.~Winter.
\newblock Covering dimension for nuclear {$\C$}-algebras.
\newblock {\em J. Funct. Anal.}, 199(2):535--556, 2003.

\bibitem{orderzero}
W.~Winter and J.~Zacharias.
\newblock Completely positive maps of order zero.
\newblock {\em M\"{u}nster J. Math.}, 2:311--324, 2009.

\bibitem{nucdim}
W.~Winter and J.~Zacharias.
\newblock The nuclear dimension of {$\C$}-algebras.
\newblock {\em Adv. Math.}, 224(2):461--498, 2010.

\bibitem{zhangRR0}
S.~Zhang.
\newblock A property of purely infinite simple {$\C$}-algebras.
\newblock {\em Proc. Amer. Math. Soc.}, 109(3):717--720, 1990.

\bibitem{Zha91}
S.~Zhang.
\newblock {$K_1$}-groups, quasidiagonality, and interpolation by multiplier projections.
\newblock {\em Trans. Amer. Math. Soc.}, 325(2):793--818, 1991.

\end{thebibliography}
\bibliographystyle{abbrv}
\end{document}